\documentclass{article}

\usepackage{amsmath,amsthm,amssymb}
\usepackage{xspace}
\usepackage{graphicx}
\usepackage{xcolor}
\usepackage{hyperref}
\usepackage{multirow}
\usepackage{tikz}
\usepackage[utf8]{inputenc}
\usepackage{enumerate}
\usepackage{fullpage}
\usepackage{float}

\newtheorem{theorem}{Theorem}
\newtheorem{lemma}[theorem]{Lemma}

\newtheorem{proposition}[theorem]{Proposition}

\newtheorem{corollary}[theorem]{Corollary}
\newtheorem{claim}{Claim}[theorem]

\theoremstyle{definition}

\newtheorem{problem}[theorem]{Problem}

\newtheorem{conjecture}[theorem]{Conjecture}

\newcommand{\AY}[1]{{#1}}

\newcommand{\ra}{\rightarrow}

\newcommand{\induce}[2]{#1\langle #2 \rangle}

\newcommand{\jbj}[1]{{#1}}

\DeclareMathOperator{\Out}{Out}
\DeclareMathOperator{\In}{In}

\newcommand{\2}{\vspace{0.2cm}}

\newenvironment{subproof}{\par\noindent {\it Proof}.\ }{\hfill$\lozenge$\par\vspace{11pt}}

\title{Arc-disjoint in- and out-branchings in digraphs of independence
  number at most 2\thanks{Research supported by the Independent
    Research Fond Denmark under grant number DFF 7014-00037B, by PICS
    project DISCO and by research grant ANR DIGRAPHS n.194718.}}

\author{J. Bang-Jensen\thanks{Department of Mathematics and Computer
    Science, University of Southern Denmark, Odense, Denmark (email:
    jbj@imada.sdu.dk). \jbj{Part of this work was done while the author was visiting LIRMM, Universit\'e de
    Montpellier as well as  INRIA Sophia Antipolis. Hospitality and financial support by both is greatefully acknowledged. Ce travail a b\'en\'efici\'e d'une aide du gouvernement français, 
g\'er\'ee par l'Agence Nationale de la Recherche au titre du projet Investissements d’Avenir UCAJEDI portant la r\'ef\'erence no ANR-15-IDEX-01.}}\and S. Bessy\thanks{LIRMM, CNRS, Universit\'e de
    Montpellier, Montpellier, France (email:stephane.bessy@lirmm.fr).}\and F. Havet\thanks{CNRS, Universit\'e
    C\^ote d'Azur, I3S and INRIA, Sophia Antipolis, France (email:
    frederic.havet@inria.fr).}\and A. Yeo\thanks{Department of
    Mathematics and Computer Science, University of Southern Denmark,
    Odense, Denmark and  Department of Mathematics, University of Johannesburg, Auckland Park, 2006 South Africa (email: yeo@imada.sdu.dk). Part of this work was done while the author was visiting LIRMM, Universit\'e de
    Montpellier. Hospitality and financial support is greatefully acknowledged }}

\begin{document}

\maketitle

\begin{abstract}
We prove that every digraph of independence number at most 2 and
arc-connectivity at least 2 has an out-branching $B^+$ and an
in-branching $B^-$ which are arc-disjoint (we call such branchings good pair).
 This is best possible in
terms of the arc-connectivity as there are infinitely many strong digraphs
with independence number 2 and arbitrarily high minimum in-and
out-degrees that have good no pair. The
result settles a conjecture by Thomassen for digraphs of independence
number 2. We prove that every digraph on at most 6 vertices and arc-connectivity at least 2 has a good pair  and give an example of a 2-arc-strong digraph $D$  on 10 vertices 
with independence
number 4  that has no
good pair. We also show that there are infinitely many digraphs with
independence number 7 and arc-connectivity 2 that have no
good pair. Finally we pose a number of open problems.\\
\noindent{}{\bf Keywords:} Arc-disjoint branchings, out-branching, in-branching, digraphs of independence number 2, arc-connectivity.
\end{abstract}
\bibliographystyle{plain}

\section{Introduction}

It is a well-known result due to Nash-Williams and Tutte
\cite{nashwilliamsJLMS36,tutteJLMS36} that every $2k$-edge-connected
graph has a set of $k$ edge-disjoint spanning trees. For digraphs,
there are many possible analogues of a spanning tree. The two most
natural ones are out-branchings and in-branchings. An {\bf
  out-branching} ({\bf in-branching}) of a digraph $D$ is a spanning
tree in the underlying graph of $D$ whose edges are oriented in $D$
such that every vertex except one, called the root, has in-degree
(out-degree) one. Edmonds \cite{edmonds1973} characterized digraphs
with $k$ arc-disjoint out-branchings with prescribed roots. Lov\'asz
\cite{lovaszJCT21} gave an algorithmic proof of Edmonds' result which
implies that one can check in polynomial time, for each fixed natural
number $k$, whether a given digraph has a collection of $k$
arc-disjoint out-branchings (roots not specified). \\

No good characterization is known for digraphs having an out-branching
and an in-branching which are arc-disjoint and very likely none
exists, due to the following result.

\begin{theorem}[Thomassen~\cite{bangJCT51}]
  \label{thm:CTNPC}
It is NP-complete to decide whether a given digraph $D$ has an
out-branching and an in-branching both rooted at the same vertex such
that these are arc-disjoint.
\end{theorem}

This implies that it is also NP-complete to decide if a digraph has
any out-branching which is arc-disjoint from some in-branching, see
Theorem \ref{thm:disjBNPC}.  The same conclusion holds already for
2-regular digraphs \cite{bangDAM161a}.

Thomassen also conjectured that every digraph of sufficiently high
arc-connectivity should have such a pair of branchings. His conjecture
was for branchings with the same root, but as we show in Proposition
\ref{prop:noBranch}, the conjecture is equivalent to the following.

\begin{conjecture}[Thomassen~\cite{thomassen1989}]
  \label{conj:CTbranch}
There is a constant $C$, \AY{such that} every digraph with arc-connectivity at least $C$ has an
out-branching and an in-branching which are arc-disjoint.
\end{conjecture}

Conjecture \ref{conj:CTbranch} has been verified for semicomplete
digraphs \cite{bangJCT51} and for locally semicomplete digraphs
\cite{bangJCT102}. In both cases arc-connectivity 2 suffices. For
general digraphs the conjecture is wide open and as far as we know it
is not known whether already $C=3$ would suffice
in Conjecture \ref{conj:CTbranch} \AY{(Figure~\ref{fig:H4} below shows that $C=2$ is not sufficient)}. In this paper, we prove the
conjecture for digraphs of independence number 2, where it suffices to
have arc-connectivity at least 2. We also show that there is no lower
bound on the minimum in- and out-degree that suffices to guarantee
that a strongly connected digraph with independence number 2 has such branchings.

We provide an example to show that arc-connectivity 2 is not
sufficient to guarantee that a digraph with independence number 2 has
an out-branching rooted at a prescribed vertex $s$ which is
arc-disjoint from an in-branching rooted a prescribed vertex $t$ for
every choice of vertices $s,t$. We also show that there are infinitely
many digraphs with independence number 3 and arc-connectivity 2 which
do not have arc-disjoint \AY{out- and in-branchings, $B_s^+,B_t^-$, rooted at some given
$s$ and $t$, respectively}. Using this we construct an infinite family of digraphs
with independence number 7 and arc-connectivity 2 which have no
out-branching which is arc-disjoint from some in-branching. \jbj{We also show that every 2-arc-strong digraph on at most 6 vertices has an out-branching which is arc-disjoint from some in-branching. Finally, we pose a number of open problems. }

\section{Notation and preliminaries}

Notation not given below follows \cite{bang2018,bang2009}. The
digraphs in this paper have no loops and no multiple arcs.  Let
$D=(V,A)$ be a digraph.  Let $X,Y\subset V$ be two sets of
vertices. If $D$ contains the arc $xy$ for every choice of $x\in X$
and $y\in Y$, then we write $X\ra Y$.  If moreover, there is no arc
with tail in $Y$ and head in $X$, then we write $X\mapsto Y$.

If $D'$ is a subdigraph of $D$ and $uv$ an arc of $D$, then we denote
by $D'+uv$ the digraph with vertex set $V(D')\cup \{u,v\}$ and arc set
$A(D')\cup \{uv\}$.

For a non-empty subset $X\subset V$ we denote by $d_D^+(X)$
(resp. $d_D^-(X)$) the number of arcs with tail (resp. head) in $X$ and head
(resp. tail) in $V-X$. We call $d_D^+(X)$ (resp. $d_D^-(X)$) the {\bf out-degree}
(resp. {\bf in-degree}) of the set $X$. Note that $X$ may be just a
vertex. We will drop the subscript when the digraph is clear from the
context. We denote by $\delta^0(D)$ the minimum over all in- and
out-degrees of vertices of $D$. This is also called the minimum {\bf
  semidegree} of a vertex in $D$.  The {\bf arc-connectivity} of $D$,
denoted by $\lambda{}(D)$, is the minimum out-degree of a proper subset
of $V$. A digraph is strongly connected (or just {\bf strong}) if
$\lambda{}(D)\geq 1$.

In- and out-branchings were defined above. We denote by $B_s^+$
\AY{(respectively $B_t^-$) an out-branching rooted at $s$
(respectively an in-branching rooted at $t$)}. We also use $B^+$ or $O$
(resp. $B^-$ or $I$) to denote an out-branching (resp. in-branching) with no root
specified.  
\begin{proposition}
  \label{prop:noBranch}
  Let $R$ be a natural number. If there exists be a digraph $H$ with
  $\lambda{}(H)=R$ which has two (possibly equal) vertices $s$ and $t$
  such that $H$ has no pair of arc-disjoint branchings $B_s^+$, $B_t^-$,
  then there exists a digraph $U$ with $\lambda{}(U)=R$ which has no
  out-branching which is arc-disjoint from some in-branching.
\end{proposition}

\begin{proof}
Let $H$ as above be given. If $s=t$, then let $H'=H$ and otherwise we
obtain $H'$ by adding a copy $X$ of the complete digraph on $R$
vertices, all possible arcs from $V(X)$ to $s$ and all possible arcs
from $t$ to $V(X)$.  It is easy to check that $\lambda{}(H')=R$ and
that $H'$ has no pair of arc-disjoint branchings $B_x^+,B_x^-$ where
$x\in X$. Now if $s=t$ take $x=s=t$ and if $s\neq t$ fix one vertex
$x\in X$. Let $U$ be the digraph that we obtain from three disjoint
copies of $H'$ by identifying the copies of $x$ in these. Then
$\lambda{}(U)=\lambda{}(H')=R$ and $U$ has no pair of arc-disjoint
branchings $B_u^+,B_v^-$ for any choice of vertices $u,v$. This
follows from the fact that $U$ could only have such branchings if one copy of $H'$
would have arc-disjoint branchings $B_x^+,B_x^-$.
\end{proof}

The following is an easy consequence of Theorem \ref{thm:CTNPC} and
the construction above.

\begin{theorem}
  \label{thm:disjBNPC}
  It is NP-complete to decide if a given digraph given has an
  out-branching and in-branching which are arc-disjoint.
  \end{theorem}

A vertex $v$ of a digraph $D=(V,A)$ is an {\bf
  in-generator} (resp. {\bf out-generator}) if $v$ can be reached from (resp. can
reach) every other vertex in $V$ by a directed path. Thus a vertex $v$
is an in-generator (resp. out-generator) if and only if $v$ is the root of
some in-branching $B_v^-$ (resp. out-branching $B_v^+$) of $D$.  The set of
in-generators (resp. out-generators) of a digraph is denoted by
$\In(D)$ (resp. $\Out(D)$).

If $X\subset V$ we denote by $\induce{D}{X}$ the subdigraph of $D$
{\bf induced} by $X$, that is, the digraph whose vertex set is $X$ and
whose arc set consists of those arcs from $A$ that have both end-vertices in $X$.

A digraph is {\bf semicomplete} if it has no pair of non-adjacent
vertices. A {\bf tournament} is a semicomplete digraph with no directed cycle of length 2. We need the following results on
semicomplete digraphs. For a survey on results for semicomplete
digraphs see Chapter 2 in \cite{bang2018}.  The following is an easy
consequence of the definition of strong connectivity.
\begin{lemma}\label{lem:Instrong}
  Let $D$ be semicomplete digraph. Then the induced subdigraphs $\induce{D}{In(D)}$   and
  $\induce{D}{Out(D)}$ are  strong.
\end{lemma}

\begin{theorem}[Moon~\cite{moonCMB9}]
\label{thm:pancyclic}
Every strong semicomplete digraph \AY{on at least $3$ vertices} is pancyclic. 
In particular, every strong semicomplete digraph \AY{on at least $2$ vertices} has a hamiltonian cycle.
\end{theorem}

An {\bf independent set} in a digraph $D=(V,A)$ is a set $X\subseteq
V$ such that $\induce{D}{X}$ has no arcs. We denote by $\alpha{}(D)$
the maximum size of an independent set in $D$.

\begin{theorem}[Chen-Manalastras~\cite{chenDM44}]
\label{thm:CMthm}
Let $D$ be a strong digraph with $\alpha{}(D)=2$. Then either $D$ has
a directed hamiltonian cycle or its vertices can be covered by two
directed cycles $C_1,C_2$ such that these are either vertex disjoint
or they intersect in a subpath of both. In particular $D$ has a
directed hamiltonian path.
\end{theorem}

By a {\bf clique} in a digraph we mean an induced subdigraph which is
semicomplete. The following is a well known consequence of a result in
Ramsey theory.

\begin{theorem}\label{thm:Ramsey}
Every digraph on at least $9$ vertices contains either an independent
set of size at least $3$ or a clique of size $4$.
\end{theorem}

A {\bf good pair} (in $D$), is a pair $(I,O)$ such that $I$ is an
in-branching of $D$, $O$ is an out-branching of $D$, and $I$ and $O$
are arc-disjoint.  A {\bf good $r$-pair} (in $D$), is a good pair
$(I,O)$ such that $r$ is the root of $I$.  A {\bf good $(r,q)$-pair}
(in $D$), is a good pair $(I,O)$ such that $r$ is the root of $I$ and
$q$ is the root of $O$.

A digraph is {\bf co-bipartite} if its underlying graph is the
complement of a bipartite graph. In other words, its vertex set can be
partitioned in two sets $V_1$, $V_2$ such that $D\langle V_1\rangle$
and $D\langle V_2\rangle$ are semicomplete digraphs.

\begin{proposition}
  \label{prop:strongnotenough}
  For every natural number $k$, there are infinitely many strong
  co-bipartite digraphs with minimum semidegree at least $k$ and no
  good pair.
\end{proposition}

\begin{proof}
  Let $T_1,T_2$ be strongly connected tournaments with
  $\delta^0{}(T_i)\geq k$ for $i=1,2$ and let $T$ be obtained from
  these by adding a new vertex $v$ and all possible arcs from $v$ to
  $V(T_1)$, all possible arcs from $V(T_2)$ to $v$ and all possible
  arcs from $V(T_2)$ to $V(T_1)$, except one arc $xy$ with goes from
  $V(T_1)$ to $V(T_2)$. The result $H$ is clearly strong and does not
  have an out-branching and an in-branching, both rooted at $v$, which
  are arc-disjoint (the arc $xy$ must belong to both branchings). Now
  let $H$ be the digraph that we obtain from two copies $H',H''$ of
  $H$ by adding a 2-cycle between the two copies $v',v''$ of $v$. This
  digraph is clearly co-bipartite. Suppose $H$ has a good pair. Then,
  w.l.o.g., the root of the out-branching belongs to $H'$ and then also
  the root of the in-branching must belong to $H'$ (as the arcs
  $v'v''$ and $v''v'$ are the only arcs between the two copies of
  $H$). But that means that $H''$ has an an in-branching rooted at
  $v''$ which is arc-disjoint from an out-branching rooted at $v''$,
  contradiction.
  \end{proof}

\section{Good pairs in semicomplete digraphs}

We first consider semicomplete digraphs and derive some easy results
that will be used later.

\begin{lemma}\label{lem:non-strong}
Let $D$ be a non-strong semicomplete digraph of order at least $4$ and
let $r \in \In(D)$ and $q \in \Out(D)$ be arbitrary.  Then $D$ has a
good $(r,q)$-pair, $(I,O)$, in $D$.
\end{lemma}
\begin{proof}
Set $W_1=D\langle \Out(D)\rangle$ and $W_2=D\langle \In(D)\rangle$.

First consider the case when $|\Out(D)| \geq 2$, in which case we can
let $U=(u_1, u_2, \ldots, u_a, u_1)$ be a hamiltonian cycle in $W_1$
($a=|\Out(D)|$) such that $u_1 = q$, which exists by
Theorem~\ref{thm:pancyclic}. Let $I'$ be any in-branching in $D-W_1$
with root $r$, which exists as $\In(D) = \In(D- W_1)$. We now
construct $I$ from $I'$ by adding the arc $u_a u_1$ and every arc from
$\{u_1,u_2,\ldots,u_{a-1}\}$ to $r$. We construct $O$ by taking the
path $u_1 u_2 \ldots u_a$ and adding every arc from $u_a$ to
$V(D)\setminus \Out(D)$.  This gives us the desired good $(r,q)$-pair,
$(I,O)$, in $D$.

We may therefore assume that $|\Out(D)| = 1$ and analogously that
$|\In(D)| = 1$.  This implies that $\Out(D)=\{q\}$ and $\In(D)=\{r\}$.
As the order of $D$ is at least $4$, there exists an arc $uv$ in
$D-\{q,r\}$.  The following out-branching, $O$, and in-branching, $I$,
form the desired good $(r,q)$-pair, $(I,O)$, in $D$.

\[
A(I) = \{qv, vr,ur\} \cup \{ z r \; | \; z \in V(D)\setminus \{r,q,u,v\} \}
\]

\[
A(O) = \{qu, uv, qr\} \cup \{ q z \; | \; z \in V(D)\setminus \{r,q,u,v\} \}
\]

\end{proof}

\begin{lemma}\label{lem:util}
Let $D$ be a semicomplete digraph and $r$ be a vertex in $\In(D)$.
If there is a subdigraph $D'$ of $D$ of order at least $2$ having a
good $r$-pair, $(I',O')$, then $D$ has a good $r$-pair, $(I,O)$.
\end{lemma}
\begin{proof}
Let $(I',O')$ be a good $r$-pair of $D'$.  Assume that $V(D')\neq
V(D)$.  Since $r$ is an in-generator in $D$, there is a vertex $y\in
V(D)\setminus V(D')$ that dominates a vertex in $V(D')$.  If $y\ra
V(D')$, then let $r'$ be the root of $O'$ and $v\in V(D')\setminus
r'$. Set $I=I' + yv$ and $O=O' + yr'$.  Then $(I,O)$ is a good
$r$-pair of $D\langle V(D')\cup\{y\}\rangle$.  If $y\not\ra V(D')$,
then $y$ dominates a vertex $z_1$ in $V(D')$ and is dominates by a
vertex $z_2$ in $V(D')$.  Set $I=I' + yz_1$ and $O=O' + z_2y$. Then
$(I,O)$ is a good $r$-pair of $D\langle V(D')\cup\{y\}\rangle$.
 
We can apply this process iteratively until we obtain a good $r$-pair of $D$.
\end{proof}

A {\bf $4$-exception} is a pair $(D,a)$ such that $D$ has 4 vertices
and contains the strong tournament of order $4$ depicted
Figure~\ref{fig:ST4} (with plain arcs) and possibly one or both arcs
in $\{dc, cb\}$ (shown as dotted arcs).
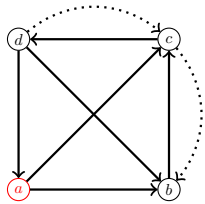
\begin{figure}[H]
\begin{center}
\tikzstyle{vertexB}=[circle,draw, minimum size=14pt, scale=0.6, inner sep=0.5pt]
\tikzstyle{vertexR}=[circle,draw, color=red!100, minimum size=14pt, scale=0.6, inner sep=0.5pt]

\begin{tikzpicture}[scale=1]
 \node (a) at (0,0) [vertexR] {$a$};
  \node (b) at (2,0) [vertexB] {$b$};
  \node (c) at (2,2) [vertexB] {$c$};
 \node (d) at (0,2) [vertexB] {$d$};
  \draw [->, line width=0.03cm] (a) to (b);
\draw [->, line width=0.03cm] (b) to (c);
\draw [->, line width=0.03cm] (c) to (d);
\draw [->, line width=0.03cm] (d) to (a);
\draw [->, line width=0.03cm] (a) to (c);
\draw [->, line width=0.03cm] (d) to (b);
\draw [->, line width=0.03cm, dotted] (d) to [in=140, out =40]  (c);
\draw [->, line width=0.03cm, dotted] (c) to [in=50, out =-50]  (b);
\end{tikzpicture}
\caption{The $4$-exceptions $(D,a)$.}\label{fig:ST4}
\end{center}
\end{figure}

\begin{proposition}
Let $D$ be a semicomplete digraph of order $4$ and let $r$ be a vertex
of $\In(D)$.  Then $D$ has a good $r$-pair unless $(D,r)$ is a
$4$-exception.
\end{proposition}

\begin{proof}
One easily sees that $D$ contains a spanning tournament $T$ such that
$r\in \In(T)$.  There are only four tournaments of order $4$, $ST_4$
the unique strong tournament of order $4$ and the three non-strong
tournaments depicted in Figure~\ref{fig:non-strong4}. For each of
these tournaments, by symmetry, we may assume that $r$ is the red
vertex and a good $r$-pair is given in Figure~\ref{fig:non-strong4}.
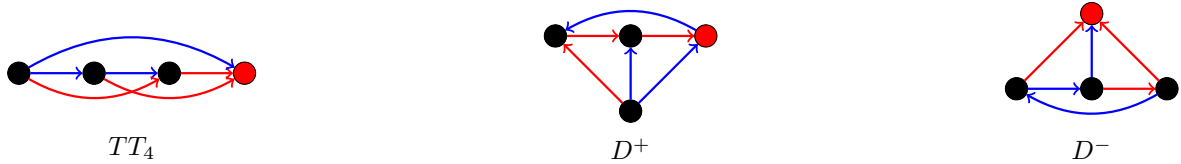
\begin{figure}[H]
\begin{center}
\tikzstyle{vertexB}=[circle,draw, top color=black!100, bottom color=black!100, minimum size=14pt, scale=0.6, inner sep=0.5pt]
\tikzstyle{vertexR}=[circle,draw, top color=red!100, bottom color=red!100, minimum size=14pt, scale=0.6, inner sep=0.5pt]
~~~
\begin{tikzpicture}[scale=1]
 \node (a1) at (0,0) [vertexB] {};
  \node (a2) at (1,0) [vertexB] {};
  \node (a3) at (2,0) [vertexB] {};
 \node (a4) at (3,0) [vertexR] {};
  \draw [->, line width=0.03cm, color=blue] (a1) to (a2);
\draw [->, line width=0.03cm, color=blue] (a2) to (a3);
\draw [->, line width=0.03cm, color=red] (a3) to (a4);
\draw [->, line width=0.03cm, color=red] (a1) to [in=-150, out=-30] (a3);
\draw [->, line width=0.03cm, color=red] (a2) to [in=-150, out=-30] (a4);
\draw [->, line width=0.03cm, color=blue] (a1) to [in=150, out=30] (a4);
\node at (1.5,-1) {$TT_4$};
\end{tikzpicture}
\hfill
\begin{tikzpicture}[scale=1]
 \node (a1) at (0,0) [vertexB] {};
  \node (a2) at (1,0) [vertexB] {};
  \node (a3) at (2,0) [vertexR] {};
 \node (b) at (1,-1) [vertexB] {};
  \draw [->, line width=0.03cm, color=red] (a1) to (a2);
\draw [->, line width=0.03cm, color=red] (a2) to (a3);
\draw [->, line width=0.03cm, color=blue] (a3) to [out=150, in=30] (a1);
\draw [->, line width=0.03cm, color=red] (b) to (a1);
\draw [->, line width=0.03cm, color=blue] (b) to (a2);
\draw [->, line width=0.03cm, color=blue] (b) to (a3);
\node at (1,-1.5) {$D^+$};
\end{tikzpicture}
\hfill
\begin{tikzpicture}[scale=1]
 \node (a1) at (0,0) [vertexB] {};
  \node (a2) at (1,0) [vertexB] {};
  \node (a3) at (2,0) [vertexB] {};
 \node (b) at (1,1) [vertexR] {};
  \draw [->, line width=0.03cm, color=blue] (a1) to (a2);
\draw [->, line width=0.03cm, color=red] (a2) to (a3);
\draw [->, line width=0.03cm, color=blue] (a3) to [out=-150, in=-30] (a1);
\draw [->, line width=0.03cm, color=red] (a1) to (b);
\draw [->, line width=0.03cm, color=blue] (a2) to (b);
\draw [->, line width=0.03cm, color=red] (a3) to (b);
\node at (1,-0.8) {$D^-$};
\end{tikzpicture}
~~~
\caption{The non-strong tournaments of order $4$ and a good $r$-pair
  when $r$ is the red vertex. The arcs of the in-branching are in red
  and the arcs of the out-branching in blue.}\label{fig:non-strong4}
 \end{center}
\end{figure}

Henceforth, we may assume that $T$ is $ST_4$.  If $r\in \{b,c,d\}$,
then there is a good $r$-pair as shown in Figure~\ref{fig:ST4good}.

\begin{figure}[H]
\begin{center}
\tikzstyle{vertexB}=[circle,draw, minimum size=14pt, scale=0.6, inner sep=0.5pt]
\tikzstyle{vertexR}=[circle,draw, color=red!100, minimum size=14pt, scale=0.6, inner sep=0.5pt]

~~~
\begin{tikzpicture}[scale=1]
 \node (a) at (0,0) [vertexB] {$a$};
  \node (b) at (2,0) [vertexR] {$b$};
  \node (c) at (2,2) [vertexB] {$c$};
 \node (d) at (0,2) [vertexB] {$d$};
  \draw [->, line width=0.03cm, color=blue] (a) to (b);
\draw [->, line width=0.03cm, color=blue] (b) to (c);
\draw [->, line width=0.03cm, color=red] (c) to (d);
\draw [->, line width=0.03cm, color=blue] (d) to (a);
\draw [->, line width=0.03cm, color=red] (a) to (c);
\draw [->, line width=0.03cm, color=red] (d) to (b);
\end{tikzpicture}
\hfill
\begin{tikzpicture}[scale=1]
 \node (a) at (0,0) [vertexB] {$a$};
  \node (b) at (2,0) [vertexB] {$b$};
  \node (c) at (2,2) [vertexR] {$c$};
 \node (d) at (0,2) [vertexB] {$d$};
  \draw [->, line width=0.03cm, color=blue] (a) to (b);
\draw [->, line width=0.03cm, color=red] (b) to (c);
\draw [->, line width=0.03cm, color=blue] (c) to (d);
\draw [->, line width=0.03cm, color=blue] (d) to (a);
\draw [->, line width=0.03cm, color=red] (a) to (c);
\draw [->, line width=0.03cm, color=red] (d) to (b);
\end{tikzpicture}
\hfill
\begin{tikzpicture}[scale=1]
 \node (a) at (0,0) [vertexB] {$a$};
  \node (b) at (2,0) [vertexB] {$b$};
  \node (c) at (2,2) [vertexB] {$c$};
 \node (d) at (0,2) [vertexR] {$d$};
  \draw [->, line width=0.03cm, color=red] (a) to (b);
\draw [->, line width=0.03cm, color=red] (b) to (c);
\draw [->, line width=0.03cm, color=red] (c) to (d);
\draw [->, line width=0.03cm, color=blue] (d) to (a);
\draw [->, line width=0.03cm, color=blue] (a) to (c);
\draw [->, line width=0.03cm, color=blue] (d) to (b);
\end{tikzpicture}
~~~
\caption{Good $r$-pairs in $ST_4$ for $r\in\{b,c,d\}$. The arcs of the
  in-branching are in red and the arcs of the out-branching in
  blue.}\label{fig:ST4good}
\end{center}
\end{figure}
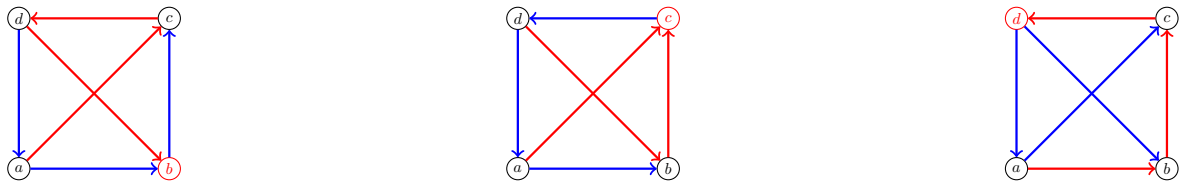

Henceforth we may assume that $r=a$.  If $D$ contains one of the arcs
$ba$, $ad$, $ca$, $bd$, then there is a good $r$-pair as shown in
Figure~\ref{fig:ST4+}.
\begin{figure}[H]
\begin{center}
\tikzstyle{vertexB}=[circle,draw, minimum size=14pt, scale=0.6, inner sep=0.5pt]
\tikzstyle{vertexR}=[circle,draw, color=red!100, minimum size=14pt, scale=0.6, inner sep=0.5pt]

~~~
\begin{tikzpicture}[scale=1]
 \node (a) at (0,0) [vertexR] {$a$};
  \node (b) at (2,0) [vertexB] {$b$};
  \node (c) at (2,2) [vertexB] {$c$};
 \node (d) at (0,2) [vertexB] {$d$};
  \draw [->, line width=0.03cm, color=blue] (a) to (b);
\draw [->, line width=0.03cm, color=blue] (b) to (c);
\draw [->, line width=0.03cm, color=red] (c) to (d);
\draw [->, line width=0.03cm, color=blue] (d) to (a);
\draw [->, line width=0.03cm] (a) to (c);
\draw [->, line width=0.03cm, color=red] (d) to (b);
\draw [->, line width=0.03cm, color=red] (b) to [in=-25, out=-155] (a);
\end{tikzpicture}
\hfill
\begin{tikzpicture}[scale=1]
 \node (a) at (0,0) [vertexR] {$a$};
  \node (b) at (2,0) [vertexB] {$b$};
  \node (c) at (2,2) [vertexB] {$c$};
 \node (d) at (0,2) [vertexB] {$d$};
  \draw [->, line width=0.03cm, color=blue] (a) to (b);
\draw [->, line width=0.03cm, color=red] (b) to (c);
\draw [->, line width=0.03cm, color=red] (c) to (d);
\draw [->, line width=0.03cm, color=red] (d) to (a);
\draw [->, line width=0.03cm, color=blue] (a) to (c);
\draw [->, line width=0.03cm] (d) to (b);
\draw [->, line width=0.03cm, color=blue] (a) to [in=-115, out=115] (d);
\end{tikzpicture}
\hfill
\begin{tikzpicture}[scale=1]
 \node (a) at (0,0) [vertexR] {$a$};
  \node (b) at (2,0) [vertexB] {$b$};
  \node (c) at (2,2) [vertexB] {$c$};
 \node (d) at (0,2) [vertexB] {$d$};
  \draw [->, line width=0.03cm] (a) to (b);
\draw [->, line width=0.03cm, color=red] (b) to (c);
\draw [->, line width=0.03cm, color=blue] (c) to (d);
\draw [->, line width=0.03cm, color=red] (d) to (a);
\draw [->, line width=0.03cm, color=blue] (a) to (c);
\draw [->, line width=0.03cm, color=blue] (d) to (b);
\draw [->, line width=0.03cm, color=red] (c) to [in=25, out=-115] (a);

\end{tikzpicture}
\hfill
\begin{tikzpicture}[scale=1]
 \node (a) at (0,0) [vertexR] {$a$};
  \node (b) at (2,0) [vertexB] {$b$};
  \node (c) at (2,2) [vertexB] {$c$};
 \node (d) at (0,2) [vertexB] {$d$};
  \draw [->, line width=0.03cm, color=blue] (a) to (b);
\draw [->, line width=0.03cm, color=red] (b) to (c);
\draw [->, line width=0.03cm, color=red] (c) to (d);
\draw [->, line width=0.03cm, color=red] (d) to (a);
\draw [->, line width=0.03cm, color=blue] (a) to (c);
\draw [->, line width=0.03cm] (d) to (b);
\draw [->, line width=0.03cm, color=blue] (b) to [in=-25, out=115] (d);
\end{tikzpicture}
~~~
\caption{Good $a$-pairs in digraphs containing $ST_4$ and an arc in
  $\{ba, ad, ca, bd\}$. The arcs of the in-branching are in red and
  the arcs of the out-branching in blue.}\label{fig:ST4+}
\end{center}
\end{figure}
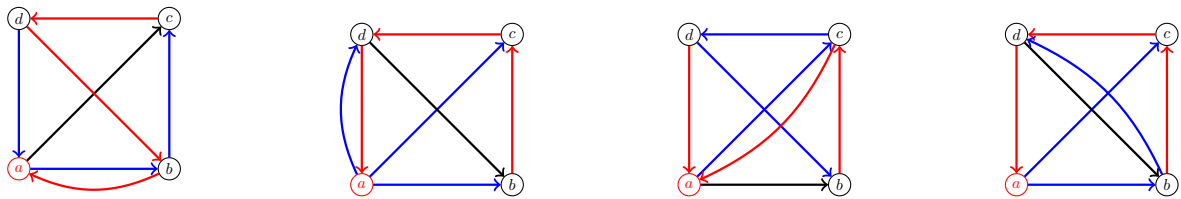
If not, then $(D,a)$ is a $4$-exception. In such a case, there is no
good $a$-pair. Indeed if there were one, then the in-branching must
contain the arcs $da$ and $cd$, and $D\setminus \{da,cd\}$ has no
out-branching because it has two sources (namely $a$ and $d$).
\end{proof}


An {\bf exception} is a pair $(D,r)$ where $D$ is a semicomplete
digraph, $r$ is a vertex of $D$, such that $N^-(r) = \{y\}$ and
$d^-(y) = 1$.

\begin{theorem}\label{nonEXCEPTION}
Let $D$ be a semicomplete digraph of order at least $4$ and let $r\in
\In(D)$.  There is a good $r$-pair if and only if $(D,r)$ is not an
exception.
\end{theorem}

\begin{proof}
If $(D,r)$ is an exception, then let $N^-(r) = \{y\}$ and $N^-(y) =
\{z\}$.  Note that any in-branching, $I$, with root $r$, must contain
the arcs $yr$ and $zy$.  However, in $D \setminus \{yr,zy\}$ we note
that both $r$ and $y$ have in-degree zero, and therefore there is no
out-branching in $D \setminus \{yr,zy\}$, which implies that there is
no good $r$-pair in $D$.

So now assume that $(D,r)$ is not an exception. If $D$ is non-strong
then we are done by Lemma~\ref{lem:non-strong}, so we may assume that
$D$ is strongly connected.  By Theorem~\ref{thm:pancyclic}, $r$ is in
a directed $3$-cycle $zyrz$.  If there exists $t\in V(D)\setminus
\{z,y,r\}$ such that $r\in \In(D\langle\{z,y,r,t\}\rangle)$ and
$(D\langle\{z,y,r,t\}\rangle, r)$ is not a $4$-exception, then, by
Lemma~\ref{lem:util}, $D$ contains a good $r$-pair.  Henceforth, we
may assume that, for all $t\in V(D)\setminus \{z,y,r\}$, either
$r\notin \In(D\langle\{z,y,r,t\}\rangle)$ or
$(D\langle\{z,y,r,t\}\rangle, r)$ is a $4$-exception.  In both cases,
$r\mapsto t$ and $y\mapsto t$.  Hence $r\mapsto V(D)\setminus
\{y,r\}$, and $y\mapsto V(D)\setminus \{r,y,z\}$.  If $r\ra y$, then
$D'=D\setminus \{yr\}$, is a semicomplete digraph with
$\In(D')=\{r\}$, and by Lemma~\ref{lem:non-strong}, $D'$ has a good
$r$-pair, which is also a good $r$-pair in $D$.  If not, then
$N^-(r)=\{y\}$ and $N^-(y)=\{z\}$, a contradiction to $(D,r)$ not
being an exception.
\end{proof}

\begin{corollary}\label{cor:in+out}
Every semicomplete digraph $D$ of order at least $4$ has a good pair.
\end{corollary}

\begin{proof}
If $D$ is non-strong the corollary follows from
Lemma~\ref{lem:non-strong}, so assume that $D$ is strongly connected.
As $D$ has order at least $4$, we note that there exists a vertex $r' \in
V(D)$ such that $d^-(r') \geq 2$.  This implies that $(D,r')$ is not
an exception and therefore there exists a good $r'$-pair $(I,O)$.
\end{proof}

In particular, notice that if $D$ is semicomplete and $\delta^0(D)\ge
2$, then either $D$ is of order at least $4$ or $D$ is a complete
digraph on $3$ vertices and in both cases, it admits a pair of
arc-disjoint in- and out-branchings.

\section{Good pairs in small digraphs}

\begin{lemma} \label{extend}
Let $D$ be a digraph and $X \subset V(D)$ be a set such that every
vertex of $X$ has both an in-neighbour and an out-neighbour in $V-X$.
If $D-X$ has a good pair then $D$ has a good pair.
\end{lemma}

\begin{proof}
Let $(I,O)$ be a good pair of $D-X$.  By assumption, every $x\in X$
has an out-neighbour $y_x$ in $V(D)\setminus X$ and an in-neighbour
$w_x$ in $V(D)\setminus X$.  Then $(I+\{xy_x \mid x\in X\}, O+ \{w_xx
\mid x\in X\})$ is a good pair for $D$.
\end{proof}

\begin{proposition}\label{prop:n=3}
Every digraph on $3$ vertices with at least $4$ arcs has a good pair.
\end{proposition}
\begin{proof}
Let $D$ be a digraph on $3$ vertices $a,b,c$ and with at least $4$
arcs.  By symmetry, and without loss of generality, $D$ has a
$2$-cycle $(a,b,a)$ and $bc$ is an arc.  Then the path $P=(a,b,c)$ is both
an in- and an out-branching.  But $Q=D\setminus A(P)$ is a path of
length $2$, which is necessarily an out- or an in-branching.  Hence
either $(P,Q)$ or $(Q,P)$ is a good pair of $D$.
\end{proof}

Let $E_4$ be the digraph depicted in Figure~\ref{fig:E4}.
\begin{figure}[H]
\begin{center}
\tikzstyle{vertexB}=[circle,draw, minimum size=14pt, scale=0.6, inner sep=0.5pt]
\tikzstyle{vertexR}=[circle,draw, color=red!100, minimum size=14pt, scale=0.6, inner sep=0.5pt]

\begin{tikzpicture}[scale=1]
 \node (a) at (0,0) [vertexB] {$y$};
  \node (b) at (2,0) [vertexB] {$x$};
  \node (c) at (2,2) [vertexB] {$y'$};
 \node (d) at (0,2) [vertexB] {$x'$};
  \draw [->, line width=0.03cm] (a) to [in=160, out =20] (b);
\draw [->, line width=0.03cm] (b) to [in=-20, out =-160] (a);
\draw [->, line width=0.03cm] (c) to [in=-20, out =-160] (d);
\draw [->, line width=0.03cm] (d) to [in=160, out =20] (c);
\draw [->, line width=0.03cm] (b) to (c);
\draw [->, line width=0.03cm] (d) to (a);
\end{tikzpicture}
\caption{The digraph $E_4$.}\label{fig:E4}
\end{center}
\end{figure}
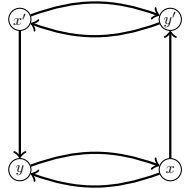

\begin{proposition}\label{prop:small4}
Let $D$ be a digraph $D$ of order $4$ with at least $6$ arcs and with
$\delta^0(D)\geq 1$.  Then $D$ has a good pair if and only if $D\neq
E_4$.
\end{proposition}
\begin{proof}
Observe that $D$ has at least as many directed $2$-cycles as it has
pairs of non-adjacent vertices.

It is easy to check that $E_4$ has no good pair.  Assume that $D\neq
E_4$ and has no good pair.  Then, by Corollary~\ref{cor:in+out}, $D$
is not semicomplete. Hence it has at least one pair of non-adjacent
vertices and thus at least one directed $2$-cycle $C$.  By
Proposition~\ref{prop:n=3} and Lemma~\ref{extend}, $D$ has no
subdigraph of order $3$ with at least $4$ arcs.  In particular, every
vertex $x$ in $V(D)\setminus V(C)$ is adjacent to at most one vertex
of $C$.  Hence $D$ contains at least two pairs of non-adjacent
vertices and thus at least two directed cycles.  Furthermore, no two
directed cycles can intersect, for otherwise their union is a digraph
of order $3$ with four arcs.  Hence $D$ has exactly two directed
$2$-cycles, $C$ and $C'$, and there are two pairs of non-adjacent vertices
forming a matching between the vertices of $C$ and $C'$.  Since $D\neq
E_4$, it must be the digraph $F_4$ depicted in Figure~\ref{fig:F4}. But, as
shown in Figure~\ref{fig:F4}, $F_4$ has a good pair.
\end{proof}

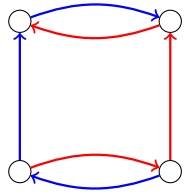
\begin{figure}[H]
\begin{center}
\tikzstyle{vertexB}=[circle,draw, minimum size=14pt, scale=0.6, inner sep=0.5pt]
\tikzstyle{vertexR}=[circle,draw, color=red!100, minimum size=14pt, scale=0.6, inner sep=0.5pt]

\begin{tikzpicture}[scale=1]
 \node (a) at (0,0) [vertexB] {};
  \node (b) at (2,0) [vertexB] {};
  \node (c) at (2,2) [vertexB] {};
 \node (d) at (0,2) [vertexB] {};
  \draw [->, color=red, line width=0.03cm] (a) to [in=160, out =20] (b);
\draw [->, color=blue, line width=0.03cm] (b) to [in=-20, out =-160] (a);
\draw [->, color=red, line width=0.03cm] (c) to [in=-20, out =-160] (d);
\draw [->, color=blue, line width=0.03cm] (d) to [in=160, out =20] (c);
\draw [->, color=red, line width=0.03cm] (b) to (c);
\draw [->, color=blue, line width=0.03cm] (a) to (d);
\end{tikzpicture}
\caption{Digraph $F_4$ and one of its good pairs. The arcs of the
  in-branching are in red and the arcs of the out-branching in
  blue.}\label{fig:F4}
\end{center}
\end{figure}

\begin{proposition}\label{prop:small}
Every digraph $D$ with $\delta^0(D)\geq 2$ and order at most $5$ has a good pair.
\end{proposition}
\begin{proof}
Since $\delta^0(D)\geq 2$, the order $n$ of $D$ is at least $3$.

If $n=3$, then $D$ is the complete digraph on three vertices, which
contains a good pair.

If $n = 4$, then $D$ has at least $8$ arcs, and so at least two
directed $2$-cycles.  A directed $2$-cycle has a good pair, so by
Lemma~\ref{extend}, $D$ has a good pair.

If $n=5$, then $D$ has at least $10$ arcs.  Hence $D$ has at least as
many directed $2$-cycles as it has pairs of non-adjacent vertices.  If
$D$ contains a semicomplete digraph $D'$ on $4$ vertices, then, by
Corollary~\ref{cor:in+out}, $D'$ has a good pair, and so by
Lemma~\ref{extend}, $D$ has a good pair.  Henceforth, we may assume
that $D$ contains no semicomplete digraph of order $4$. Thus $D$ has
at least two pairs of non-adjacent vertices and thus at least two
directed $2$-cycles.

If $D$ contains a subdigraph on $3$ vertices with $4$ arcs, then this
digraph has a good pair by Proposition~\ref{prop:n=3}, and so $D$ has
a good pair by Lemma~\ref{extend}.  Henceforth we may assume that $D$
contains no subdigraph on $3$ vertices with $4$ arcs.  But this is
impossible, indeed if this was the case, then all directed $2$-cycles are
vertex disjoint, so there are at most two of them, and each directed $2$-cycle
is incident to at least three non-edges. This contradicts the fact
that the number of pairs of non-adjacent vertices is no greater than
the number of directed $2$-cycles.
\end{proof}

\begin{lemma}\label{lem:3good} If $D=(V,A)$ is a digraph on $n$
    vertices with $\lambda{}(D)=2$ that contains a subdigraph on $n-3$
    vertices with a good pair, then $D$ has a good pair. In
    particular, if $D$ has $6$ vertices and contains a subdigraph on $3$
    vertices which has a least $4$ arcs then $D$ has a good pair.
  \end{lemma}
  \begin{proof}
    First, notice that the second part of the statement follows from
    Proposition~\ref{prop:n=3}. Now, let $X\subset V$ be a subset of
    size $n-3$ such that $\induce{D}{X}$ has a good pair and let
    $V-X=\{a,b,c\}$. If some vertex $v\in\{a,b,c\}$ has both an
    in-neighbour in $X$, then $\induce{D}{X+v}$ has a good pair and
    then the claim follows from Lemma \ref{extend}, so we can assume
    there is no such vertex $v$. Then there cannot exist two vertices of
    $V-X$ with an in-neighbour in $X$ and two vertices of $V-X$ with an
    out-neighbour in $X$. As $\lambda(D)\ge 2$, we may assume
    w.l.o.g. that $a$ has two in-neighbours in $x_1,x_2\in X$ and also
    that $c$ has an out-neighbour $x'_1$ in $X$. This implies that
    $ab,ac$ are arcs of $D$ and also $bc$ as $c$ has no in-neighbour
    in $X$. Suppose first that $b$ has no in-neighbour in $X$, then
    $cb$ is also an arc and now we can extend the good pair of
    $\induce{D}{X}$ by adding the arcs $x_1a,ac,cb$ to the
    out-branching of $\induce{D}{X}$ and adding the arcs $ab,bc,cx'_1$
    to the in-branching of $\induce{D}{X}$. In the case when  $b$ has an
    in-neighbour $x$ in $X$, we can extend the good pair of
    $\induce{D}{X}$ by adding the arcs $x_1a,xb,ac$ to the
    out-branching of $\induce{D}{X}$ and adding the arcs $ab,bc,cx'_1$
    to the in-branching of $\induce{D}{X}$
  \end{proof}

\section{Good pairs in co-bipartite digraphs}

\begin{theorem} \label{cobipartite}
Let $D$ be a co-bipartite digraph with $\lambda(D) \geq 2$.  Then $D$
has a good pair.
\end{theorem}

\begin{proof}
Let $D$ be a co-bipartite digraph with vertex partition $(V_1,V_2)$,
that is $D_i=D\langle V_i\rangle$ is a semicomplete digraph for
$i=1,2$.  Without loss of generality, we may assume $|V_1|\leq |V_2|$.
If $|V(D)|\leq 5$, then we have the result by
Proposition~\ref{prop:small}.  Therefore we may assume $|V(D)|\geq 6$.
We distinguish several cases.

{\bf Case 1:} $|V_1|\ge 4$. Let $a_1a_2$ be an arc from $\In(D_1)$ to
$V_2$ in $D$, which exists as $D$ is strongly connected and there is
no arc from $\In(D_1)$ to $V_1\setminus \In(D_1)$.

First assume that $(D_1,a_1)$ is an exception and denote by $y_1$ the
unique in-neighbour of $a_1$ in $D_1$ and by $z_1$ the unique
in-neighbour of $y_1$ in $D$. By Corollary~\ref{cor:in+out}, as
$D_2$ contains at least four vertices, it admits a good pair
$(I_2,O_2)$. Moreover, let $I_1$ be an in-branching of $D_1$ rooted at
$a_1$, and $O_1$ be the out-branching of $D_1\setminus y_1$ containing
all the arcs leaving $a_1$. Now, as $a_1$ and $y_1$ have in-degree at
least 2 in $D$ they respectively have an in-neighbour $a_1'$ and
$y_1'$ in $V_2$.  Then $(I_2 +a_1a_2 + I_1,O_2+a_1'a_1+y_1'y_1+ O_1)$
is a good pair of $D$.

 Assume now that $(D_1,a_1)$ is not an exception. By
 Theorem~\ref{nonEXCEPTION}, $D_1$ admits a good $a_1$-pair
 $(I_1,O_1)$. We shall find a similar pair for $D_2$. As
 $\lambda(D)\ge 2$, $D\setminus a_1a_2$ is strong and so, there exists
 an arc $b_1b_2$ of $D\setminus a_1a_2$ from $V_1$ to $\Out(D_2)$.
 Consider the digraph $\tilde{D}$ obtained from $D$ by reversing all
 its arcs, and set $\tilde{D_2} = \tilde{D}\langle V_2\rangle$. As
 $\In(\tilde{D_2})=\Out(D_2)$, $b_2$ is a vertex of
 $\In(\tilde{D_2})$. If $(\tilde{D_2}, b_2)$ is an exception, then we
 conclude as previously that $\tilde{D}$ has a good pair. Thus, $D$
 has also a good pair. Otherwise, if $(\tilde{D_2}, b_2)$ is not an
 exception, by Theorem~\ref{nonEXCEPTION}, $\tilde{D_2}$ admits a good
 $b_2$-pair $(\tilde{O_2},\tilde{I_2})$. It means that $D_2$ admits a
 good pair $(I_2,O_2)$ such that $b_2$ is the root of out-branching
 $O_2$. In this case, $(I_2+a_1a_2+ I_1, O_1+b_1b_2+ O_2)$ is a
 good pair of $D$.

{\bf Case 2:} $|V_1|\le 2$.
 Then $|V_2|\geq 4$ since $|V(D)|\geq 6$.
Hence, by Corollary~\ref{cor:in+out}, $D_2$ has a good pair.
Thus, by Lemma~\ref{extend}, $D$ has also a good pair.

{\bf Case 3:} $|V_1| = 3$. If $D_2$ admits a good pair, then we
conclude by Lemma~\ref{lem:3good}.  So we can assume that $D_2$ has
no good pair.  By Corollary~\ref{cor:in+out} we have $|V_2|=3$. If
$D_1$ has a good pair, then we apply Lemma\ref{lem:3good}.  Therefore
we may assume that $D_1$ has also no good pair.  By
Proposition~\ref{prop:n=3}, $D_1$ and $D_2$ are tournaments.  This
implies that each vertex of $V_i$ is incident to at least two
  arcs whose other end-vertex is in $V_{3-i}$ for all $i\in [2]$.

If $D$ contains a semicomplete subdigraph $D'$ of order $4$, then by
Corollary~\ref{cor:in+out}, $D'$ has a good pair, and so by
Lemma~\ref{extend}, $D$ has a good pair.  Thus we may assume that $D$
has no such subdigraph. In particular each vertex of $V_i$ is
non-adjacent to at least one vertex in $V_{3-i}$ for all $i\in [2]$.
Suppose that some vertex in $v_1\in V_1$ forms a 2-cycle with a vertex
$v_2\in V_2$. Then we can assume, by Lemma \ref{lem:3good}, that $v_i$
is non-adjacent to the two other vertices of $V_{3-i}$ for $i=1,2$. By
the remark above $D-\{v_1,v_2\}$ is not semicomplete so it contains
two non-adjacent vertices which are respectively in $V_1$ and
$V_2$. As these vertices have in- and out-degree at least 2, we
conclude that each vertex of $V_i$ forms a 2-cycle with one vertex of
$V_{3-i}$ for $i=1,2$. Furthermore, we can assume that none of these
2-cycles have a vertex in common, otherwise we conclude with
Lemma~\ref{lem:3good}. Now we see that $D$ is one of the two digraphs
in Figure \ref{fig:6vertex} in which we show a good pair for $D$.

\begin{figure}[H]
\begin{center}
\tikzstyle{vertexB}=[circle,draw, minimum size=14pt, scale=0.6, inner sep=0.5pt]
\tikzstyle{vertexR}=[circle,draw, color=red!100, minimum size=14pt, scale=0.6, inner sep=0.5pt]
\begin{tikzpicture}[scale=0.8]
  \node (a1) at (0,4) [vertexB] {$a_1$};
  \node (b1) at (0,2) [vertexB] {$b_1$};
  \node (c1) at (0,0) [vertexB] {$c_1$};
  \node (a2) at (2,4) [vertexB] {$a_2$};
  \node (b2) at (2,2) [vertexB] {$b_2$};
  \node (c2) at (2,0) [vertexB] {$c_2$};
  \draw[->,  line width=0.03cm, color=red]  (a1) to (b1);
  \draw[->,  line width=0.03cm, color=red]  (b2) to (c2);
  \draw[->,  line width=0.03cm, color=blue]  (a2) to (b2);
  \draw[->,  line width=0.03cm, color=blue]  (b1) to (c1);
  \draw [->, line width=0.03cm, color=blue] (a1) to  [in=150, out=30] (a2);
  \draw [->, line width=0.03cm, color=red] (b1) to  [in=150, out=30] (b2);
  \draw [->, line width=0.03cm, color=blue] (c1) to  [in=150, out=30] (c2);
  \draw [->, line width=0.03cm, color=red] (a2) to  [in=-30, out=210] (a1);
  \draw [->, line width=0.03cm, color=blue] (b2) to  [in=-30, out=210] (b1);
  \draw [->, line width=0.03cm, color=red] (c1) to [in=210, out=150] (a1);
  \draw [->, line width=0.03cm] (c2) to  [in=-30, out=210] (c1);
  \draw [->, line width=0.03cm] (c2) to [in=-30, out=30] (a2);
  
  \node (a1) at (6,4) [vertexB] {$a_1$};
  \node (b1) at (6,2) [vertexB] {$b_1$};
  \node (c1) at (6,0) [vertexB] {$c_1$};
  \node (a2) at (8,4) [vertexB] {$a_2$};
  \node (b2) at (8,2) [vertexB] {$b_2$};
  \node (c2) at (8,0) [vertexB] {$c_2$};
  \draw[->,  line width=0.03cm, color=red]  (a1) to (b1);
  \draw[->,  line width=0.03cm, color=red]  (b1) to (c1);
  \draw[->,  line width=0.03cm, color=red]  (b2) to (a2);
  \draw[->,  line width=0.03cm, color=blue]  (c2) to (b2);
  \draw[->, line width=0.03cm, color=blue]  (a2) to [in=30, out=-30] (c2);
  \draw [->, line width=0.03cm, color=blue] (a1) to  [in=150, out=30] (a2);
  \draw[->, line width=0.03cm, color=blue] (b2) to [out=210, in=-30] (b1);
  \draw[->, line width=0.03cm, color=blue] (c2) to [out=210, in=-30] (c1);
  \draw[->, line width=0.03cm, color=red] (c1) to [out=30,in= 150]   (c2);
  \draw[->, line width=0.03cm, color=red] (a2) to [out=210,in= -30]   (a1);
  \draw[->, line width=0.03cm] (b1) to [out=30, in=150] (b2);
  \draw[->, line width=0.03cm] (c1) to [out=150, in= 210] (a1);

\end{tikzpicture}
\caption{The two possible digraphs and good pairs in these when there
  are three 2-cycles.}\label{fig:6vertex}
\end{center}
\end{figure}
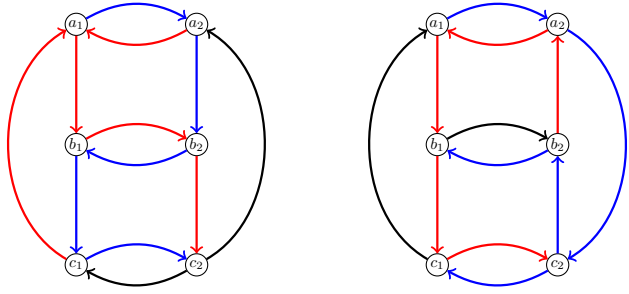
Therefore we can assume that $D$ has no directed $2$-cycle and hence one can
label the vertices of $V_1$ by $a_1, b_1, c_1$ and the vertices of
$V_2$ by $a_2, b_2, c_2$ so that $a_1$(resp. $b_1$, $c_1$) is not
adjacent to $a_2$ (resp. $b_2$, $c_2$) and adjacent to the two other
vertices and for every $i, j, k\in [2]$, $D\langle \{a_i, b_{j},
c_{k}\}\rangle$ is a tournament.

In particular, every vertex vertex of $D$ has in- and out-degree exactly $2$.
 
 If $D_1$ and $D_2$ are both directed $3$-cycles, then $D\setminus
 (A(D_1)\cup A(D_2))$ is a directed $6$-cycle $C$.  Let $a$ be an arc
 of $C$ from $V_1$ to $V_2$, let also $P_1$ be a hamiltonian directed
 path of $D_1$ ending in the tail of $a$ and $P_2$ a hamiltonian
 directed path of $D_2$ starting at the head of $a$. Then $(C\setminus a, P_1+
 a + P_2)$ is a good pair of $D$.

 If one of the $D_i$, say $D_1$, is not a directed cycle, then $D$
 must be one of the three digraphs depicted in
 Figure~\ref{fig:order6}, and so $D$ has a good pair.
 \end{proof}

 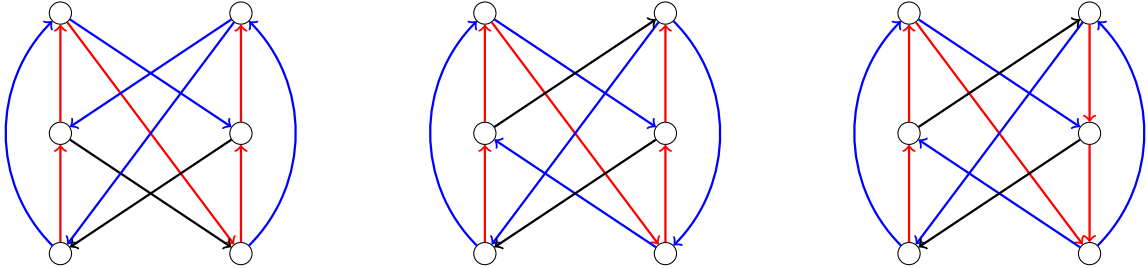
\begin{figure}[H]
\begin{center}
\tikzstyle{vertexB}=[circle,draw, minimum size=14pt, scale=0.6, inner sep=0.5pt]
\tikzstyle{vertexR}=[circle,draw, color=red!100, minimum size=14pt, scale=0.6, inner sep=0.5pt]

~~~
\begin{tikzpicture}[scale=0.8]
 \node (a1) at (0,0) [vertexB] {};
  \node (b1) at (0,2) [vertexB] {};
  \node (c1) at (0,4) [vertexB] {};
 \node (a2) at (3,0) [vertexB] {};
  \node (b2) at (3,2) [vertexB] {};
  \node (c2) at (3,4) [vertexB] {};

  \draw [->, line width=0.03cm, color=red] (a1) to (b1);
 \draw [->, line width=0.03cm, color=red] (b1) to (c1);
 \draw [->, line width=0.03cm, color=blue] (a1) to [in=-135, out=135] (c1);

 \draw [->, line width=0.03cm, color=red] (a2) to (b2);
 \draw [->, line width=0.03cm, color=red] (b2) to (c2);
 \draw [->, line width=0.03cm, color=blue] (a2) to [in=-45, out=45] (c2);

 \draw [->, line width=0.03cm, color=red] (c1) to (a2);
 \draw [->, line width=0.03cm, color=blue] (c1) to (b2);
 \draw [->, line width=0.03cm, color=black] (b1) to (a2);
 \draw [->, line width=0.03cm, color=blue] (c2) to (a1);
 \draw [->, line width=0.03cm, color=blue] (c2) to (b1);
 \draw [->, line width=0.03cm, color=black] (b2) to (a1);
 \end{tikzpicture}
\hfill
\begin{tikzpicture}[scale=0.8]
 \node (a1) at (0,0) [vertexB] {};
  \node (b1) at (0,2) [vertexB] {};
  \node (c1) at (0,4) [vertexB] {};
 \node (a2) at (3,0) [vertexB] {};
  \node (b2) at (3,2) [vertexB] {};
  \node (c2) at (3,4) [vertexB] {};

  \draw [->, line width=0.03cm, color=red] (a1) to (b1);
 \draw [->, line width=0.03cm, color=red] (b1) to (c1);
 \draw [->, line width=0.03cm, color=blue] (a1) to [in=-135, out=135] (c1);

 \draw [->, line width=0.03cm, color=red] (a2) to (b2);
 \draw [->, line width=0.03cm, color=red] (b2) to (c2);
 \draw [->, line width=0.03cm, color=blue] (c2) to [out=-45, in=45] (a2);

 \draw [->, line width=0.03cm, color=red] (c1) to (a2);
 \draw [->, line width=0.03cm, color=blue] (c1) to (b2);
 \draw [->, line width=0.03cm, color=blue] (a2) to (b1);
 \draw [->, line width=0.03cm, color=blue] (c2) to (a1);
 \draw [->, line width=0.03cm, color=black] (b1) to (c2);
 \draw [->, line width=0.03cm, color=black] (b2) to (a1);
 \end{tikzpicture}
 \hfill
\begin{tikzpicture}[scale=0.8]
 \node (a1) at (0,0) [vertexB] {};
  \node (b1) at (0,2) [vertexB] {};
  \node (c1) at (0,4) [vertexB] {};
 \node (a2) at (3,0) [vertexB] {};
  \node (b2) at (3,2) [vertexB] {};
  \node (c2) at (3,4) [vertexB] {};

  \draw [->, line width=0.03cm, color=red] (a1) to (b1);
 \draw [->, line width=0.03cm, color=red] (b1) to (c1);
 \draw [->, line width=0.03cm, color=blue] (a1) to [in=-135, out=135] (c1);

 \draw [->, line width=0.03cm, color=red] (b2) to (a2);
 \draw [->, line width=0.03cm, color=red] (c2) to (b2);
 \draw [->, line width=0.03cm, color=blue] (a2) to [in=-45, out=45] (c2);

 \draw [->, line width=0.03cm, color=red] (c1) to (a2);
 \draw [->, line width=0.03cm, color=blue] (c1) to (b2);
 \draw [->, line width=0.03cm, color=blue] (a2) to (b1);
 \draw [->, line width=0.03cm, color=blue] (c2) to (a1);
 \draw [->, line width=0.03cm, color=black] (b1) to (c2);
 \draw [->, line width=0.03cm, color=black] (b2) to (a1);
 \end{tikzpicture}
~~~
\caption{Good pairs of three co-bipartite digraphs of order $6$.}\label{fig:order6}
\end{center}
\end{figure}

\section{Good pairs in $2$-arc-strong digraphs with $\alpha{}(D)\leq 2$}

Now we are ready to prove our main result.
\begin{theorem} \label{mainX}
If $D$ is a digraph with $\alpha(D) \leq 2 \leq \lambda(D)$, then $D$
has a good pair.
\end{theorem}

\begin{proof}
For the sake of contradiction, assume that $D$ has no good pair.  By
Proposition~\ref{prop:small}, $|V(D)|\geq 6$.

\begin{claim}\label{claimA}
There is no $Q \subseteq V(D)$ such that $\induce{D}{Q}$ has a good
pair and every vertex in $V(D) \setminus Q$ is adjacent to at least
one vertex in $Q$.
\end{claim}
\begin{subproof}
Suppose to the contrary that there exists $Q \subseteq V(D)$ such that
$\induce{D}{Q}$ has a good pair $(I',O')$, and every vertex in $V(D)
\setminus Q$ is adjacent to at least one vertex in $Q$. Furthermore
assume that $|Q|$ is maximum with this property. Let $X = N_D^+(Q)$
and let $Y=N_D^-(Q)$.  By the maximality of $Q$ and
Lemma~\ref{extend}, $X \cap Y = \emptyset$.

Let $X_i$, $i\in [a]$, be the terminal strong components in
$\induce{D}{X}$ and let $Y_j$, $j\in [b]$, be the initial strong
components in $\induce{D}{Y}$.  As $\alpha(D) \leq 2$ we note that $1
\leq a,b \leq 2$.

As $\lambda(D) \geq 2$ there are at least two arcs from $X_i$ (for
each $i \in[a]$) to $Y$ and at least two arcs from $X$ to $Y_j$ (for
each $j \in [b]$). Let $x_1 y$ be an arbitrary arc out of $X_1$ ($ y
\in Y$). If $y \in V(Y_1) \cup V(Y_b)$, then without loss of
generality assume that $y \in Y_1$.  Let $P_1 = \{ x_1y \}$.  There
now exists an arc, $x y_1$ from $X$ to $Y_1$ which is different from
$x y_1$ (as $Y_1$ has at least two arcs into it) and we take $P_2 =
\{x y_1\}$.  If $a=2$, then we let $x_2 y'$ be any arc out of $X_2$,
which is different from $x y_1$ and we add $x_2y'$ to $P_1$.  If
$b=2$, then we let $x' y_2$ be any arc into $Y_2$ which is different
from $x_2 y'$ (and which by the definition of $x_1 y$ is also
different from $x_1 y$) and we add $x' y_2$ to $P_2$.

Let $D_X$ be the digraph obtained from $\induce{D}{X}$ by adding one
new vertex $y^*$ and arcs from $x_i$ to $y^*$ for $i \in [a]$. Note
that $\In(D_X) = \{y^*\}$ and that there therefore exists an
in-branching $I_X$ in $D_X$ with root $y^*$.  Set $T_X=I_X-y^*$.
Analogously let $D_Y$ be equal to $\induce{D}{Y}$ after adding one new
vertex $x^*$ and arcs from $x^*$ to $y_i$ for $i \in [b]$. Note that
$\Out(D_Y) = \{x^*\}$ and that there therefore exists an out-branching
$O_Y$ in $D_Y$ with root $x^*$.  Set $T_Y=O_Y-x^*$.

Now let $I$ be the in-branching of $D$ obtained from $I'$ as
follows. For each $u \in Y$ we add $u$ and an arc from $u$ to $Q$ to
$I$. We then add to $I$ the digraph $T_X$ and the arcs in $P_1$.

Let $O$ be the out-branching of $D$ obtained from $O'$ as follows: For
each $u \in X$ we add $u$ and an arc from $Q$ to $u$ to $O$; we then
add to $O$ the digraph $T_Y$ and the arcs in $P_2$.  By construction,
$I$ and $O$ are arc-disjoint so $(I,O)$ is a good pair in $D$, a
contradiction.
\end{subproof}

\begin{claim}\label{claimB}
There is no $Q \subseteq V(D)$, such that $\induce{D}{Q}$ is not
semicomplete but does have a good pair.
\end{claim}
\begin{subproof}
 Suppose to the contrary that there exists a $Q \subseteq V(D)$, such
 that $\induce{D}{Q}$ is not semicomplete but does have a good pair
 $(I',O')$.  Let $u$ and $v$ be non-adjacent vertices in $Q$. Let $w
 \in V(D) \setminus Q$ be arbitrary.  Note that $w$ is adjacent to $u$
 or to $v$ (or both) as $\alpha(D) \leq 2$. Therefore every vertex in
 $V(D) \setminus Q$ is adjacent to at least one vertex in $Q$. This
 contradicts Claim~\ref{claimA}.
\end{subproof}

Let $R$ be a largest clique in $D$. 

\begin{claim}\label{claim1}
$|R| = 3$.
\end{claim}
\begin{subproof}
Using Ramsey theory, it is well-known that every digraph of order at
least 6 either has an independent set of size 3 or a clique of order
3.  As $\alpha(D) \leq 2$, we have $|R| \geq 3$.  Suppose to the
contrary that $|R| \geq 4$.

Let $X = N_D^+(R)$ and let $Y=N_D^-(R)$ and let $Z=V(D) \setminus (R
\cup X \cup Y)$.  For the sake of contradiction, assume that there
exists $w \in X \cap Y$.  By Corollary~\ref{cor:in+out} and
Lemma~\ref{extend}, there exists a good pair in $\induce{D}{R \cup
  \{w\}}$. Furthermore $\induce{D}{R \cup \{w\}}$ is not semicomplete,
as $R$ is a largest clique. This contradicts Claim~\ref{claimB}. So $X
\cap Y = \emptyset$.

We distinguish two cases depending on whether or not $\induce{D}{R}$
is strongly connected.

\2

{\bf Case A.} {\em $\induce{D}{R}$ is strongly connected.}

\2

Assume that some vertex $x \in X$ has two arcs, say $r_1x$ and $r_2x$,
into it from $R$.  We will first show that there is either a good
$r_1$-pair or a good $r_2$-pair in $\induce{D}{R}$.  Without loss of
generality $r_1$ is an in-neighbour of $r_2$.  Assume for a
contradiction that there is no good $r_1$-pair and no good $r_2$-pair
in $\induce{D}{R}$, which by Theorem~\ref{nonEXCEPTION} implies that
both $(R,r_1)$ and $(R,r_2)$ are exceptions.  Therefore
$N_{\induce{D}{R}}^-(r_1) = \{r'\}$ for some $r' \in V(R)$ and
$d_R^-(r') = 1$ and $N_R^-(r_2) = \{r_1\}$ and
$d_{\induce{D}{R}}^-(r_1) = 1$. This implies that $|R|=3$,
contradiction.  This contradiction implies that there is a good
$r_1$-pair or a good $r_2$-pair in $\induce{D}{R}$.  Without loss of
generality assume that there is a good $r_1$-pair, $(I_R,O_R)$, in
$\induce{D}{R}$.  Then $(I_R+r_1x, O_R+r_2x)$ is a $x$-good pair in
$\induce{D}{R \cup \{x\}}$.  Furthermore $\induce{D}{R \cup \{x\}}$ is
not semicomplete, as $R$ is a largest clique. This contradicts
Claim~\ref{claimB}.  Therefore every $x \in X$ has exactly one arc
into it from $R$.  Analogously every $y \in Y$ has exactly one arc
out of it to $R$.

Consequently, every vertex in $V(D) \setminus R$ is adjacent to at
most one vertex in $R$. Therefore, if $u,v \in V(D) \setminus R$ are
non-adjacent in $D$ then there exists $r \in R$ such that $\{u,v,r\}$
is an independent set, a contradiction.  Hence $V(D) \setminus R$ is
semicomplete and $D$ is co-bipartite. Thus, by
Theorem~\ref{cobipartite}, $D$ has a good pair, a contradiction. This
completes the proof of Case~A.

\2

{\bf Case B.} {\em $\induce{D}{R}$ is not strongly connected.}

\2

Let $R_1,R_2,\ldots , R_l$ denote the strong components of
$\induce{D}{R}$, where $l \geq 2$, such that $R_i \mapsto R_j$ for all
$1 \leq i < j \leq l$.  As $\lambda(D) \geq 2$, there are at least two
arcs out of $R_l$ in $D$. Let $r_l x$ and $r_l' x'$ be two such arcs
and note that $x,x' \in X$.  Analogously let $y r_1$ and $y' r_1'$
denote two arcs from $Y$ to $R_1$. 
By Lemma \ref{lem:non-strong}, $R$ has a good pair so it follows from the the maximality of $R$ and Claim \ref{claimB} $x,x'$ have no out-neighbour in $V(R)$ and $y,y'$ have no in-neighbour in $V(R)$.

Assume that there exists an arc $r x$ from $R$ to $x$ which is
distinct from $r_l x$.  By Lemma~\ref{lem:non-strong}, there exists a
good $r_l$-pair $(I_R,O_R)$ of $\induce{D}{R}$, and so $(I_R+r_lx,
O_R+rx)$ is a good pair in $\induce{D}{R \cup \{x\}}$. This
contradicts Claim~\ref{claimB}.

Therefore $x$ is adjacent with exactly one vertex in $R$ (namely
$r_l$).  Analogously $x'$, $y$ and $y'$ are each adjacent with exactly
one vertex in $R$.  As $\alpha(D)\leq 2$, this implies that $D' = \induce{D}{\{x,x',y,y'\}}$
is a digraph of order $4$ which must be  semicomplete. So by
  Corollary~\ref{cor:in+out}, there exists a good pair $(I',O')$ in
  $D'$. Moreover by Lemma~\ref{lem:non-strong}, $\induce{D}{R}$ has a
  good $(r_l,r_1)$-pair $(I_R,O_R)$.  Now $(I'
  \cup I_R+ r_l x, O' \cup O_R + y r_1)$ is a
  good pair for $\induce{D}{R \cup V(D')}$. This contradicts
  Claim~\ref{claimB}.
%


\end{subproof}

\begin{claim}\label{claimC}
No subdigraph of $D$ of order at least $4$ has a good pair.
\end{claim}
\begin{subproof}
Suppose to the contrary that a subdigraph $D'$ of $D$ of order at
least $4$ has a good pair.  If $D'$ is semicomplete, then it
contradicts Claim~\ref{claim1}, and if $D'$ is not semicomplete, then
it contradicts Claim~\ref{claimB}.
\end{subproof}

\begin{claim}\label{claim2}
All directed $2$-cycles are vertex disjoint.
\end{claim}
\begin{subproof}
Suppose to the contrary that two directed $2$-cycles intersect, say
$(x,y,x)$ and $(x,z,x)$ are both directed $2$-cycles in $D$. Let $X = \{x,y,z\}$.  Let
$I=(y,x,z)$ and $O=(z,x,y)$.  Note that $(I,O)$ is a good pair for
$\induce{D}{X}$.

Let $w \in V(D) \setminus X$ be arbitrary. The vertex $w$ is not
adjacent to both $x$ and $y$ for otherwise by
Proposition~\ref{prop:n=3}, $\induce{D}{\{x,y,w\}}$ has a good pair
and so by Lemma~\ref{extend}, $\induce{D}{\{x,y,z,w\}}$ has a good
pair, a contradiction to Claim~\ref{claimC}.  Similarly, $w$ is not
adjacent to both $x$ and $z$.

Assume now for a contradiction that $w$ is adjacent to $y$ and $z$.
We will again show that $\induce{D}{\{x,y,z,w\}}$ has a good pair, a
contradiction to Claim~\ref{claimC}.  If $wy,zw \in A(D)$ or $yw,wz
\in A(D)$, then $\induce{D}{X \cup \{w\}}$ has a good pair by
Lemma~\ref{extend}, because $\induce{D}{X}$ has a good pair.  If $wy,wz
\in A(D)$, then $((w, y, x,z) ; (w,z,x,y))\}$ is a good pair of
$\induce{D}{X \cup \{w\}}$.  If $yw,zw \in A(D)$, then $((y,x,z,w) ;
(z,x,y,w))$ is a good pair of $\induce{D}{X \cup \{w\}}$.  Therefore
$w$ is adjacent to at most one vertex in $X$.

If $w,w' \in V(D) \setminus X$, then $w$ and $w'$ must be adjacent,
since otherwise there is a vertex in $X$ which together with
$\{w,w'\}$ forms an independent set of size 3, a contradiction.
Therefore $D - X$ is a semicomplete digraph.

Now $\induce{D}{X}$ is not semicomplete for otherwise $D$ is
co-bipartite, a contradiction to Theorem~\ref{cobipartite}.  But
$(I,O)$ is a good pair of $\induce{D}{X}$, a contradiction to
Claim~\ref{claimB}.
\end{subproof}

\begin{claim}\label{claim3}
$\induce{D}{R}$ is a tournament.
\end{claim}
\begin{subproof} 
Suppose for a contradiction that $\induce{D}{R}$ contains a $2$-cycle
$(x,y,x)$.  Let $R=\{x,y,z\}$.

We distinguish several case depending on the arcs between $z$ and
$\{x,y\}$.

\2

{\bf Case A.}  $zx,zy \in A(D)$.

\2

As $\lambda(D) \geq 2$ there are at least two arcs leaving the set
$\{x,y\}$ in $D$.  Let $u_1v_1$ and $u_2v_2$ be two such arcs. By
Claim~\ref{claim2} we have $z \not\in \{v_1,v_2\}$.

If $v_1=v_2$, then, without loss of generality, we may assume that
$u_1v_1 = xv$ and $u_2v_2 = yv$ (where $v=v_1=v_2$).  Now let
$I'=(z,x,y,v)$ and $O'=(z,y,x,v)$ and note that $(I',O')$ is a good
pair for $\induce{D}{R \cup \{v\}}$. But $\induce{D}{R \cup \{v\}}$ is
not semicomplete by our choice of $R$. This contradicts
Claim~\ref{claimB}.  So $v_1 \neq v_2$.

As $\lambda(D) \geq 2$ there are at least two arcs entering $z$ in
$D$.  Let $r_1z$ and $r_2z$ be two such arcs. Note that $r_1 \not=
r_2$ and that by Claim~\ref{claim2} we have $r_1,r_2\notin \{x,y\}$.

Let $I=(z,x,y)$ and $O=(z,y,x)$ and note that $(I,O)$ is a good pair
in $R$.  By Lemma~\ref{extend} and Claim~\ref{claimC} we note that
there are no arcs from $R$ to $\{r_1,r_2\}$ and there are no arcs from
$\{v_1,v_2\}$ to $R$.  Therefore $Y=\{r_1,r_2,v_1,v_2\}$ is a set of
four distinct vertices.

We will now show that $Y$ is a clique, a contradiction to the
maximality of $R$. We will do this by showing that every vertex in $Y$
has at most one neighbour in $R$. This will imply the claim as
$\alpha{}(D)=2$.

\begin{itemize}
\item If a vertex in $\{r_1,r_2\}$ is adjacent to a vertex in
  $\{x,y\}$ then assume without loss of generality that $r_1$ is
  adjacent to $x$, which implies that $r_1 x \in A(D)$, by the above
  observation.  Let $I' = \{r_1x,xy,zx\}$ and $O'=\{r_1z,zy,yx\}$ and
  note that $(I',O')$ is a good pair for $\induce{D}{X \cup \{r_1\}}$,
  contradicting Claim~\ref{claimC}.  Therefore no vertex from
  $\{r_1,r_2\}$ is adjacent to a vertex in $\{x,y\}$ and so the
  vertices $r_1$ and $r_2$ are adjacent to exactly one vertex in $R$.

\item Now assume for the sake of contradiction that a vertex of
  $\{v_1, v_2\}$, say $v_1$, is adjacent to at least two vertices in
  $R$. The vertex $v_1$ is not adjacent to $x$ and $y$, for otherwise,
  we could have let $v_1=v_2$, contradicting the arguments above.  So
  we may assume that $v_1$ is adjacent to $x$ and $z$, so $xv_1, z v_1
  \in A(D)$.  Let $I' = (z,y,x,v_1)$ and $O'= (z, x,y) + zv_1$ and
  note that $(I',O')$ is a good pair for $\induce{D}{X \cup \{v_1\}}$,
  contradicting Claim~\ref{claimC}. Therefore every vertex in
  $\{v_1,v_2\}$ is adjacent to at most one vertex in $R$.
\end{itemize}
This completes the Case A. 

\2 The case when $xz,yz \in A(D)$ is proved analogously to Case A by
reversing all arcs.

\2

{\bf Case B.}  $zx,yz \in A(D)$.

\2

Let $I=(z,x)+yx$ and $O=(x,y,z)$ and note that $(I,O)$ is a good pair
for $R$.  Let $w \in V(D) \setminus R$ be arbitrary. Note that $w$
cannot have an arc to $R$ and an arc from $R$ by Lemma~\ref{extend}
and Claim~\ref{claimC}. For the sake of contradiction assume that $w$
is adjacent to at least two of the vertices in $R$. If $w$ is adjacent
to $x$ and $y$, then we would be in the above case (as $wx,wy \in
A(D)$ or $xw,yw \in A(D)$). We may, without loss of generality, assume
that $wz,wx \in A(D)$.  Let $I'=(w,z,x) + yx$ and $O'=(w,x,y,z)$ and
note that $(I',O')$ is a good pair for $\induce{D}{R \cup \{w\}}$
contradicting Claim~\ref{claimC}. Therefore $w$ has at most one
neighbour in $R$ and so $V(D)\setminus R$ is a clique.  Hence $D$ is
co-bipartite, a contradiction to Theorem~\ref{cobipartite}. This
completes the Case B.

\2

The case when $xz,zy \in A(D)$ is proved analogously to Case B by
reversing all arcs.  This completes the proof of Claim~\ref{claim3}.
\end{subproof}

\begin{claim}\label{claimNO2C}
Either $D$ has no directed $2$-cycles or $D$ has exactly one directed
$2$-cycle and $|V(D)|=7$.
\end{claim}
\begin{subproof}
Assume for a contradiction that $D$ has a directed $2$-cycle
$(x,y,x)$. Let $X$ (resp. $Y$) be the set of vertices in $V(D)\setminus
\{x,y\}$ adjacent to $x$ (resp. $y$).  Since $\delta^0(D)\geq 2$, $x$
(resp. $y$) has an in-neighbour and an out-neighbour in $X$
(resp. $Y$) and they cannot be the same by Claim~\ref{claim2}.  So
$|X|,|Y|\geq 2$.  By Claim~\ref{claim3}, $X \cap Y = \emptyset$. As no
vertex in $X$ is adjacent to $y$ (resp. $x$), $X$ (resp. $Y$) is a
clique and so $X\cup \{x\}$ (resp. $Y\cup\{y\}$) is a clique, implying
that $|X|=|Y|=2$ by Claim~\ref{claim1}.

Let $Z=V(D)\setminus (\{x,y\}\cup X\cup Y)$.  Then $Z$ is non-empty
for otherwise $D$ would be co-bipartite and hence have a good pair by
Theorem \ref{cobipartite}.  Now $X\cup Z$ and $Y\cup Z$ are cliques
since $\alpha(D)\leq 2$, so $|Z|=1$ by Claim~\ref{claim1}.  Hence $D$
has 7 vertices. Moreover we check that every pair of vertices of
  $D$ except $\{x,y\}$ is in the neighbourhood of a third vertex, and
  then cannot induce a directed 2-cycle in $D$ by Claim~\ref{claim3}.
\end{subproof}

Now we are ready to finish the proof of Theorem \ref{mainX}.  Note
that $n\geq 7$ since every digraph $D$ on at most 6 vertices with
$\alpha{}(D)=2$, $\delta^0{}(D)\geq 2$ and no 2-cycle is co-bipartite.
It also follows from Theorem \ref{thm:Ramsey} and Claim \ref{claim1}
that $n\leq 8$.

By Theorem \ref{thm:CMthm}, $D$ has a directed hamiltonian path
$P$. Let $x$ and $y$ be its initial and terminal vertex,
respectively. Let $D'$ be the digraph that we obtain by deleting all
the arcs of $P$. If $D'$ has precisely one initial strong component,
then it has an out-branching $B^+$ so $(P,B^+)$ is a good
pair. Similarly, if $D'$ has only one terminal strong component, then
it has an in-branching $B^-$ and $(B^-,P)$ is a good pair. Hence we
may assume that $D'$ is not strongly connected and that it has at
least two initial components $D'_1,D'_2$ and at least two terminal
components $D'_3,D'_4$. As $\delta^0{}(D)\geq 2$ we have
$\delta^0{}(D')\geq 1$, implying that $|D'_i|\geq 3$ for $i\in [4]$ if
$D$ has no directed $2$-cycle. Since $n\leq 8$ and $n=7$ if there is
precisely one directed $2$-cycle, $D'$ has exactly two strong
components $D'_1,D'_2$ and there are no arcs between these. This means
that every arc of $D$ that goes between $V(D'_1)$ and $V(D'_2)$
belongs to $P$.\\

Since $n\geq 7$ we may assume w.l.o.g. that $|D'_1|\geq 4$. By Claim
\ref{claimC} and Corollary~\ref{cor:in+out}, $D\langle V(D'_1)\rangle$
is not semicomplete. The digraph $D'_2$ has also order at least $2$
and if it has order at least 4 then $D\langle V(D'_2)\rangle$ is not
semicomplete.
\medskip

Suppose first that $|D'_1|=|D'_2|=4$ (in which case $D$ has no directed 2-cycle
by Claim \ref{claimNO2C}).  W.l.o.g. $x\in V(D'_1)$ so $x$ has
in-degree at least 2 in $D'_1$, implying that $D'_1$ has precisely 5
arcs (it cannot have more since then it would be semicomplete,
contradicting Claim \ref{claim1}) and hence $P$ uses no arc in
$D\langle V(D'_1)\rangle$.  Let $x^+$ be the successor of $x$ on $P$
and note that $x^+\in V(D'_2)$.

Let us first observe that $D'_1$ has an out-branching $B^+_1$ that
does not use all arcs out of $x$. This is clear if $D'_1$ is
hamiltonian so assume it is not. Then $D'_1$ is the digraph with
vertex set $z_1,z_2,z_3,z_4$ and arcs $z_1z_3,
z_1z_4,z_2z_1,z_3z_2,z_4z_2$.  Now let $B^+_{x^+}$ be an out-branching
of $D'_2$ rooted at $x^+$ and let $xz$ be an arc out of $x$ in $D'_1$
which is not in $B^+_1$. Then we obtain a good pair $(I,O)$ by letting
$I=P\setminus xx^+ + xz$ and $O=B^+_1\cup B^+_{x^+}+xx^+$, a
contradiction.

\medskip

Assume now that $|D'_2|=3$. Then $V(D'_2)$ is a clique, and so
$D\langle V(D'_2)\rangle$ has no directed $2$-cycles by
Claim~\ref{claim3}. Thus $D'_2=D\langle V(D'_2)\rangle$, $D'_2$ is a
directed 3-cycle and $P$ does not use any arc inside $D\langle
V(D'_2)\rangle$. Label the vertices of $V(D'_2)$ by $a,b,c$ so that
$P$ visits these vertices in that order. Let $a^+$ (resp. $b^+$) be
the successor of $a$ (resp. $b$) on $P$.  If $D'_2$ is the directed
3-cycle $(a,c,b,a)$, then we obtain a good pair $(I,O)$ (and a
contradiction) by letting $I=P\setminus aa^++ac$ (so the root will be
$y$) and $O=(c,b,a,a^+)\cup B^+_{a^+}$, where $B^+_{a^+}$ is any out
branching rooted at $a^+$ in $D_1'$.  If $D_2$ is the 3-cycle $(a,b,c,a)$
then we obtain a good pair $(I,O)$ (and a contradiction) by letting
$I=P\setminus bb^+ +bc$ and $O=(c,a,b,b^+)\cup B^+_{b^+}$, where
$B^+_{b^+}$ is any out-branching rooted at $b^+$ in $D_1'$.

\medskip

Assume finally that $|D'_2|=2$.  Then $D'_2$ is a directed $2$-cycle
$(a,b,a)$.  Without loss of generality, we may assume that $P$ visits
$a$ before $b$.  Let $a^+$ be successor of $a$ on $P$.  Letting
$I=P\setminus aa^+ +ab$ and $O=(b,a,a^+)\cup B^+_{a^+}$, where
$B^+_{a^+}$ is any out-branching rooted at $a^+$ in $D_1'$.  Then
$(I,O)$ is a good pair of $D$, a contradiction.  This completes the
proof of Theorem \ref{mainX}.
\end{proof}

\section{Digraphs with bounded independence number  and no good pair}

The following example shows that $\alpha(D)=2=\lambda{}(D)$ is not
sufficient to guarantee a pair of arc-disjoint branchings
$B_s^+,B_t^-$ for every choice of vertices $s,t\in V(D)$.  Let $H_1$
be the strong semicomplete digraph on four vertices $a_1,b_1,c_1,d_1$
that we obtain from the directed 4-cycle $(a_1,b_1,c_1,d_1,a_1)$ by adding the arcs
of the directed 2-cycle $(a_1,c_1,a_1)$ and the arc $d_1b_1$. Let $H_2$ be the
strong semicomplete digraph on four vertices $a_2,b_2,c_2,d_2$ that we
obtain from the directed 4-cycle $(a_2,d_2,c_2,b_2,a_2)$ by adding the arcs of the
directed 2-cycle $(a_2,c_2,a_2)$ and the arc $b_2d_2$. The digraph $W$ is obtained
from the disjoint union of $H_1$ and $H_2$ by adding the arcs of the directed 4-cycle $(d_1,d_2,b_1,b_2,d_1)$. See
Figure \ref{fig:W}. It is easy to verify that $D$ is 2-arc-strong.

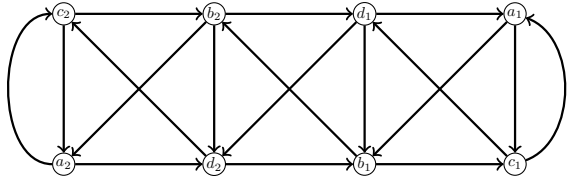
\begin{figure}[H]
\begin{center}
\tikzstyle{vertexB}=[circle,draw, minimum size=14pt, scale=0.6, inner sep=0.5pt]
\tikzstyle{vertexR}=[circle,draw, color=red!100, minimum size=14pt, scale=0.6, inner sep=0.5pt]

\begin{tikzpicture}[scale=1]
  \node (c2) at (0,2) [vertexB] {$c_2$};
  \node (a2) at (0,0) [vertexB] {$a_2$};
  \node (b2) at (2,2) [vertexB] {$b_2$};
  \node (d2) at (2,0) [vertexB] {$d_2$};
  \node (d1) at (4,2) [vertexB] {$d_1$};
  \node (b1) at (4,0) [vertexB] {$b_1$};
  \node (a1) at (6,2) [vertexB] {$a_1$};
  \node (c1) at (6,0) [vertexB] {$c_1$};
  \draw [->, line width=0.03cm] (c2) to (b2);
  \draw [->, line width=0.03cm] (b2) to (d1);
  \draw [->, line width=0.03cm] (d1) to (a1);
  \draw [->, line width=0.03cm] (a2) to (d2);
  \draw [->, line width=0.03cm] (d2) to (b1);
  \draw [->, line width=0.03cm] (b1) to (c1);
  \draw [->, line width=0.03cm] (b2) to (a2);
  \draw [->, line width=0.03cm] (b2) to (d2);
  \draw [->, line width=0.03cm] (d1) to (d2);
  \draw [->, line width=0.03cm] (a1) to (b1);
  \draw [->, line width=0.03cm] (c1) to (d1);
  \draw [->, line width=0.03cm] (b1) to (b2);
  \draw [->, line width=0.03cm] (d2) to (c2);
  \draw [->, line width=0.03cm] (d1) to (b1);
  \draw [->, line width=0.03cm] (c2) to (a2);
  \draw [->, line width=0.03cm] (a1) to (c1);
  \draw [->, line width=0.03cm] (a2) to [in=-180, out =-180] (c2);
  \draw [->, line width=0.03cm] (c1) to [in=-20, out = 20] (a1);
\end{tikzpicture}
\caption{The 2-arc-strong digraph $W$}\label{fig:W}
\end{center}
\end{figure}

\begin{proposition}\label{prop:W}
  The digraph $W$ has no pair of arc-disjoint branchings $B_{c_2}^+,B_{c_1}^-$.
\end{proposition}

\begin{proof}
  Suppose that such branchings do exist.  We first consider the case
  when the arc $c_2b_2$ is in $B_{c_2}^+$. Then the arc $c_2a_2$ is in
  $B_{c_1}^-$ and the arc $b_2a_2$ is in $B_{c_2}^+$. The set
  $\{c_2,a_2\}$ shows that the arc $a_2d_2$ is in $B_{c_1}^-$ and the
  set $\{c_2,a_2,d_2\}$ shows that $d_2b_1$ is in $B_{c_1}^-$. Now the
  set $\{c_2,a_2,b_2,d_2\}$ shows that the arc $b_2d_1$ is in
  $B_{c_2}^+$. This implies that the arc $b_2d_2$ is in $B_{c_1}^-$ .
  Next the set $\{a_2,b_2,c_2,d_2,b_1\}$ shows that the arc $b_1c_1$
  must belong to $B_{c_1}^-$ and the set $\{a_2,b_2,c_2,d_2,d_1,b_1\}$
  shows that the arc $d_1a_1$ must be in $B_{c_2}^+$. Then arc
  $a_1c_1$ must belong to $B_{c_2}^+ $, the arc $a_1b_1$ must belong
  to $B_{c_1}^-$ and the arc $d_1b_1$ must belong to $B_{c_2}^+$. Now
  all out-going arcs of $d_1$ were added to $B_{c_2}^+$,
  contradiction.

  Suppose next that the arc $c_2a_2$ belongs to $B_{c_2}^+$. Then
  analogously to the argument above we conclude that the arcs
  $c_2b_2,b_2d_1$ all belong to $B_{c_1}^-$ and the arcs
  $a_2d_2,d_2b_1,b_1c_1$ all belong to $B_{c_2}^+$. This implies that
  the arc $b_1b_2$ belongs to $B_{c_2}^+$ but then both arcs leaving
  the vertex $b_1$ are in in $B_{c_2}^+$, contradiction.
  \end{proof}

  \begin{proposition}\label{prop:infalpha3}
    There are infinitely many digraphs with arc-connectivity $2$ and
    independence number $3$ which do not have arc-disjoint branchings
    $B_s^+,B_t^-$ for some choice of vertices $s,t\in V$.
  \end{proposition}

  \begin{proof}
\jbj{For $n\geq 9$ let ${\cal W}'_n$  be the class of digraphs that we obtain from a strong semicomplete digraph $S$ on $n-8$ vertices and a copy of the digraph $W$ above by adding all possible arcs from $V(S)$ to $c_2$ and all possible arcs from $c_1$ to $V(S)$. It is easy to check that every digraph in ${\cal W}'_n$ is 2-arc-strong and has independence number 3. We claim that no digraph in ${\cal W}'_n$ has pair of arc-disjoint branchings $B^+_s,B^-_t$ where $s,t\in V(S)$. Suppose that such a digraph $W'_n$ had arc-disjoint branchings $B^+_s,B^-_t$. Then the restriction of these branchings to $V(W)$ would be an out-branching rooted at $c_2$ and an in-branching rooted at $c_1$ which are arc-disjoint, contradicting Proposition \ref{prop:W}.
    }
  \end{proof}

  \jbj{The following result  shows that there is no function $f(k)$ with the
  property that every 2-arc-strong digraph $D=(V,A)$ with $\alpha(D)=k$ and
  $|V|\geq f(k)$ has a good pair.}
  
\begin{theorem}
    \label{thm:nopairB}
    There exist infinitely many digraphs with arc-connectivity $2$ and
    independence number at most $7$ which have no good pair.  \end{theorem}
\begin{proof}
\jbj{Let $S$ be an arbitrary strong semicomplete digraph on $n\geq 1$ vertices and let $W_S$ be the digraph on $n+24$ vertices that we obtain from $S$ and three copies of the digraph $W$ from Proposition \ref{prop:W} by adding all possible arcs from $V(S)$ to the three copies of the vertex $c_2$ and all possible arcs from the three copies of the vertex $c_1$ to $V(S)$. Then $W_s$ is 2-arc-strong and has independence number 7. We claim that $W_s$ has no out-branching which is arc-disjoint from some in-branching. Suppose such a pair $B^+_s,B^-_t$ did exist. Then at least one copy of $W$ would contain none of $s,t$ and hence the restriction of $B^+_s,B^-_t$ to that copy would be  pair of arc-disjoint branchings $B_{c_2}^+,B_{c_1}^-$, contradicting that $W$ has no such pair. }
\end{proof}

  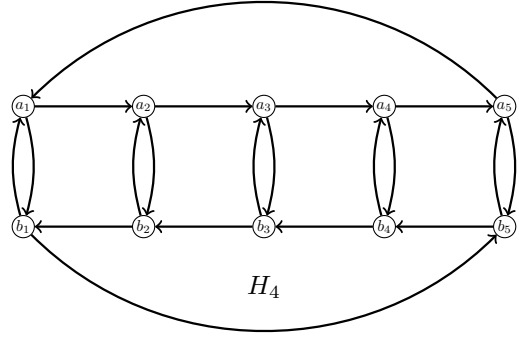
\begin{figure}[H]
\begin{center}
\tikzstyle{vertexB}=[circle,draw, minimum size=14pt, scale=0.6, inner sep=0.5pt]
\tikzstyle{vertexR}=[circle,draw, color=red!100, minimum size=14pt, scale=0.6, inner sep=0.5pt]

\begin{tikzpicture}[scale=0.8]
  \node (a1) at (0,2) [vertexB] {$a_1$};
  \node (b1) at (0,0) [vertexB] {$b_1$};
  \node (a2) at (2,2) [vertexB] {$a_2$};
  \node (b2) at (2,0) [vertexB] {$b_2$};
  \node (a3) at (4,2) [vertexB] {$a_3$};
  \node (b3) at (4,0) [vertexB] {$b_3$};
  \node (a4) at (6,2) [vertexB] {$a_4$};
  \node (b4) at (6,0) [vertexB] {$b_4$};
  \node (a5) at (8,2) [vertexB] {$a_5$};
  \node (b5) at (8,0) [vertexB] {$b_5$};
  \draw [->,line width=0.03cm] (a1) to (a2);
  \draw [->,line width=0.03cm] (a2) to (a3);
  \draw [->,line width=0.03cm] (a3) to (a4);
  \draw [->,line width=0.03cm] (a4) to (a5);
  \draw [->,line width=0.03cm] (b5) to (b4);
  \draw [->,line width=0.03cm] (b4) to (b3);
  \draw [->,line width=0.03cm] (b3) to (b2);
  \draw [->,line width=0.03cm] (b2) to (b1);
  \draw[->, line width=0.03cm] (a5) [out=135,in=45] to (a1);
  \draw[->, line width=0.03cm] (b1) [out=-45,in=225] to (b5);
  \draw[->, line width=0.03cm] (a1) [out=-75,in=75] to (b1);
  \draw[->, line width=0.03cm] (b1) [out=105,in=255] to (a1);
  \draw[->, line width=0.03cm] (a2) [out=-75,in=75] to (b2);
  \draw[->, line width=0.03cm] (b2) [out=105,in=255] to (a2);
  \draw[->, line width=0.03cm] (a3) [out=-75,in=75] to (b3);
  \draw[->, line width=0.03cm] (b3) [out=105,in=255] to (a3);
  \draw[->, line width=0.03cm] (a4) [out=-75,in=75] to (b4);
  \draw[->, line width=0.03cm] (b4) [out=105,in=255] to (a4);
  \draw[->, line width=0.03cm] (a5) [out=-75,in=75] to (b5);
  \draw[->, line width=0.03cm] (b5) [out=105,in=255] to (a5);
  \node () at (4,-1) {$H_4$};

\end{tikzpicture}
\caption{The digraph $H_4$ with $\alpha{}(H_4)=4$, $\lambda{}(H_4)=2$
  and no good pair.}\label{fig:H4}
\end{center}
\end{figure}

\begin{proposition}
  The digraph $H_4$ in Figure \ref{fig:H4} has $\alpha{}(H_4)=4$,
  $\lambda{}(H_4)=2$ and no good pair.
\end{proposition}

\begin{proof}
  First observe that for each $i\in [5]$ the subdigraph
  $H_4\langle \{a_i,a_{i+1},b_i,b_{i+1}\}\rangle$, where $a_6=a_1,b_6=b_1$ induces
  a copy of the digraph $E_4$. By Proposition \ref{prop:small4}, $E_4$
  has no good pair. Now suppose that $H_4$ has a good pair
  $(I,O)$. Then these branchings must avoid at least one arc inside
  each of the five copies of $E_4$ (otherwise the restriction of $(I,O)$
  to such a copy would be a good pair in $E_4$). But $H_4$ has 20 arcs,
  18 of which must belong to either $I$ or $O$ and there is pair of arcs with at least one in each of the 
  five copies of $E_4$, contradiction.
  \end{proof}

  \begin{proposition}
    \label{prop:n=6OK}
    Every digraph on 6 vertices and arc-connectivity at least $2$ has a good pair.
  \end{proposition}

  \begin{proof}
    Let $D$ have 6 vertices and $\lambda{}(D)\geq 2$ and suppose that
    $D$ has no good pair.  By Theorem \ref{mainX} we may assume that
    $\alpha{}(D)\geq 3$. If $\alpha{}(D)=4$ then $D$ contains, as a
    spanning subdigraph, the digraph that we obtain from the complete
    bipartite graph $K_{2,4}$ by replacing each edge by a directed
    2-cycle and it is easy to check that this has a good pair. So we
    can assume that $\alpha{}(D)=3$. Let $X=\{x_1,x_2,x_3\}$ be an
    independent set of size 3. Then each $x_i$ is incident to a
    directed 2-cycle $(x_i,y_i,x_i)$. By Lemma \ref{lem:3good}, we can assume that
    $D$ does not contain a subdigraph on 3 vertices with a good pair,
    so by Proposition \ref{prop:n=3} we conclude that
    $|\{y_1,y_2,y_3\}|=3$ and that $d^+(x_i)=d^-(x_i)=2$ for $i\in
    \AY{[3]}$. \AY{ Let $A'$ contain all arcs between $X$ and $V(D)\setminus X$ except the arcs $x_iy_i$ and $y_ix_i$ for $i=1,2,3$.
    Note that $|A'|=6$.  
    If the arcs of $A'$} form a directed 6-cycle, then without loss of
    generality this is the 6-cycle $x_1y_2x_3y_1x_2y_3x_1$ and now $D$
    contains the spanning subdigraph $D'$ in Figure \ref{fig:6ok}
    where we only show the arcs of a good pair. So we may assume that
    $V(D)\setminus X$ has a vertex with in-degree 2 and another with out-degree 2
    wrt the arcs $A'$. W.l.o.g. $y_1$ has in-degree 2 and $y_3$ has
    out-degree 2 wrt $A'$. Now $D$ contains the spanning subdigraph
    $D''$ shown in the right part of Figure \ref{fig:6ok} together
    with a good pair.

    \begin{figure}[H]
\begin{center}
\tikzstyle{vertexB}=[circle,draw, minimum size=14pt, scale=0.6, inner sep=0.5pt]
\tikzstyle{vertexR}=[circle,draw, color=red!100, minimum size=14pt, scale=0.6, inner sep=0.5pt]

\begin{tikzpicture}[scale=0.7]
  \node (y1) at (2,0) [vertexB] {$y_1$};
  \node (x2) at (0,2) [vertexB] {$x_2$};
  \node (x3) at (4,2) [vertexB] {$x_3$};
  \node (y3) at (0,4) [vertexB] {$y_3$};
  \node (x1) at (2,6) [vertexB] {$x_1$};
  \node (y2) at (4,4) [vertexB] {$y_2$};
  \draw [->,line width=0.03cm, color=red] (y3) to (x1);
  \draw [->,line width=0.03cm, color=blue]  (x1) to (y2);
  \draw [->,line width=0.03cm,color=red] (y2) to (x3);
  \draw [->,line width=0.03cm,color=red] (x3) to (y1);
  \draw [->,line width=0.03cm, color=blue]  (x2) to (y3);
  \draw [->,line width=0.03cm, color=blue]  (x1) [out=240,in=120] to (y1);
  \draw [->,line width=0.03cm, color=blue]  (y2) [out=180,in=60] to (x2);
  \draw [->,line width=0.03cm, color=blue]  (y3)  to (x3);
  \draw [->,line width=0.03cm, color=red]  (y1) [out=60,in=-60] to (x1);
  \draw [->,line width=0.03cm, color=red]  (x2) [out=30,in=210] to (y2);
  \node () at (2,-1) {$D'$};
  
  \node (yy1) at (7,1) [vertexB] {$y_1$};
  \node (xx2) at (9,5) [vertexB] {$x_2$};
  \node (xx3) at (11,5) [vertexB] {$x_3$};
  \node (yy3) at (11,1) [vertexB] {$y_3$};
  \node (xx1) at (7,5) [vertexB] {$x_1$};
  \node (yy2) at (9,1) [vertexB] {$y_2$};
  \node () at (9,-1) {$D''$};
  \draw [->,line width=0.03cm, color=red] (yy1) to (xx1);
  \draw [->,line width=0.03cm, color=red] (xx1) to (yy2);
  \draw [->,line width=0.03cm, color=red] (yy3) to (xx3);
  \draw [->,line width=0.03cm, color=red] (xx2) to (yy2);
  \draw [->,line width=0.03cm, color=red] (xx3) to (yy1);
  \draw [->,line width=0.03cm, color=blue] (yy3) to (xx1);
  \draw [->,line width=0.03cm, color=blue] (yy2) to (xx3);
  \draw [->,line width=0.03cm, color=blue] (xx1) [out=210,in=150]to (yy1);
  \draw [->,line width=0.03cm, color=blue] (yy2) [out=75,in=-75] to (xx2);
  \draw [->,line width=0.03cm, color=blue] (xx3) [out=-45, in=45] to (yy3);
\end{tikzpicture}
\caption{The two possible digraphs $D'$ and $D''$ with good
  pairs.}\label{fig:6ok}
\end{center}
\end{figure}
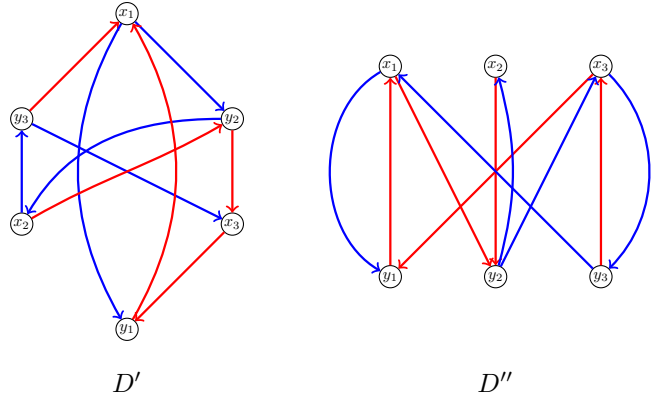
    \end{proof}

\section{Remarks and open problems}

\begin{problem}
What is the smallest number $n$ of vertices in a 2-arc-strong digraph which has no good pair?
\end{problem}

By Proposition \ref{prop:n=6OK} and the example $H_4$ in Figure \ref{fig:H4} we know that $7\leq n\leq 10$.\\

\jbj{The infinite family ${\cal W}'_n$ in the proof of Proposition \ref{prop:infalpha3} shows that there are infinitely many 2-arc-strong digraphs with independence number 3 which have only  a linear number of pairs $s,t$ for which arc-disjoint branchings $B_s^+,B_t^-$ exists (for each $W'_n$ with $n\geq 10$ we can take $t\in V(S)$  arbitrary and let $s=b_1$). This leads to the following question.}

\jbj{\begin{problem}
Does  there exist a digraph with independence number $3$ and arc-strong connectivity $2$ without a good pair?
\end{problem}}
    
\begin{conjecture}\label{conj:2asalpha2}
  Every 2-arc-strong digraph $D=(V,A)$ with $\alpha{}(D)=2$ has a pair of arc-disjoint branchings
  $B_s^+,B_s^-$ for every choice of $s\in V$.
\end{conjecture}

Figure \ref{fig:badmulti} shows that Conjecture \ref{conj:2asalpha2} does not hold for directed multigraphs (the example is Figure 4 in \cite{bangDM309}).

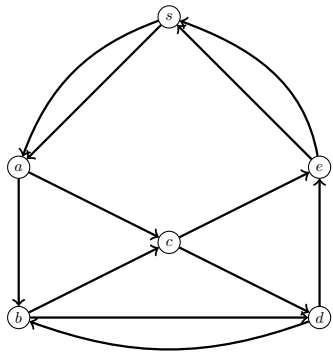
\begin{figure}[H]
\begin{center}
\tikzstyle{vertexB}=[circle,draw, minimum size=14pt, scale=0.6, inner sep=0.5pt]
\tikzstyle{vertexR}=[circle,draw, color=red!100, minimum size=14pt, scale=0.6, inner sep=0.5pt]

\begin{tikzpicture}[scale=1]
\node (s) at (2,4) [vertexB] {$s$};
\node (a) at (0,2) [vertexB]{$a$};
\node (b) at (0,0) [vertexB]{$b$};
\node (c) at (2,1) [vertexB]{$c$};
\node (d) at (4,0) [vertexB]{$d$};
\node (e) at (4,2) [vertexB]{$e$};

\draw [->, line width=0.03cm] (a) to (b);
\draw [->, line width=0.03cm] (a) to (c);
\draw [->, line width=0.03cm] (b) to (c);
\draw [->, line width=0.03cm] (c) to (d);
\draw [->, line width=0.03cm] (c) to (e);
\draw [->, line width=0.03cm] (d) to (e);
\draw [->, line width=0.03cm] (s) to (a);
\draw [->, line width=0.03cm] (e) to (s);
\draw [->, line width=0.03cm] (b) to (d);
\draw [->, line width=0.03cm] (d) to [out=200,in=-20] (b);
\draw [->, line width=0.03cm] (e) to [out=100,in=-20] (s);
\draw [->, line width=0.03cm] (s) to [out=200,in=70] (a);

\end{tikzpicture}
\end{center}
\caption{A 2-arc-strong multigraph $D$ with $\alpha{}(D)=2$ and no pair of arc-disjoint branchings $B_s^+,B_s^-$.}\label{fig:badmulti}
\end{figure}

\begin{conjecture}
  Every 3-arc-strong digraph $D=(V,A)$ with $\alpha{}(D)=2$ has a pair of arc-disjoint branchings
  $B_s^+,B_t^-$ for every choice of $s,t\in V$.
\end{conjecture}

\begin{problem}
  \label{prob:disjB}
  What is the complexity of deciding whether a digraph $D=(V,A)$ with
  $\alpha{}(D)=2$ has an out-branching and an in-branching that are
  arc-disjoint?
\end{problem}

\begin{problem}
\label{prob:disjBroots}
  What is the complexity of deciding whether a digraph $D=(V,A)$ with $\alpha{}(D)=2$  has an out-branching $B_s^+$ and and in-branching $B_t^-$ that are arc-disjoint when $s,t\in V$ are prescribed? 
\end{problem}

It was shown in \cite{fradkinJCT110} that one can decide in polynomial time whether a digraph of independence number 2 has arc-disjoint paths $P_1,P_2$, where $P_i$ is an $(s_i,t_i)$-path for $i=1,2$, where $s_1,s_2,t_1,t_2$ are part of the input. This suggests that Problems \ref{prob:disjB} and \ref{prob:disjBroots} could be polynomial-time solvable.\\

\noindent{}{\bf Acknowledgment:} The authors thank Carsten Thomassen for interesting discussions on arc-disjoint in- and out-branchings in digraphs of bounded independence number.

\end{document}